\documentclass[11pt]{amsart}

\usepackage{amsmath}
\usepackage{amsfonts}
\usepackage{epsf,amssymb,times,overpic,amscd}
\usepackage[usenames]{color}
\usepackage[all]{xy}

\usepackage{graphicx}
\usepackage{epsfig}
\usepackage{epstopdf}

\newtheorem{thm}{Theorem}[section]
\newtheorem{prop}[thm]{Proposition}
\newtheorem{lemma}[thm]{Lemma}
\newtheorem{cor}[thm]{Corollary}

\theoremstyle{definition}
\newtheorem{defn}[thm]{Definition}

\theoremstyle{remark}

\newtheorem{rmk}[thm]{Remark}

\newtheorem{prob}[thm]{Problem}

\newcommand{\Z}{\mathbb{Z}}

\newcommand{\K}{\mathbf{k}}
\newcommand{\pa}{\partial}
\newcommand{\Hom}{\operatorname{Hom}}

\newcommand{\End}{\operatorname{End}}

\newcommand{\Ind}{\operatorname{Ind}}

\newcommand{\Ext}{\operatorname{Ext}}

\newcommand{\n}{\noindent}
\newcommand{\ot}{\otimes}
\newcommand{\od}{\odot}

\newcommand{\ra}{\rightarrow}

\newcommand{\mf}{\mathbf}

\newcommand{\cal}{\mathcal}

\newcommand{\wt}{\widetilde}
\newcommand{\ov}{\overline}

\newcommand{\mb}{\mathbf{1}}
\newcommand{\nh}{\widetilde{NH}{}}
\newcommand{\nhmn}{\widetilde{NH}{}_n^m}

\newcommand{\vzn}{v_n^0}
\newcommand{\nhzn}{\widetilde{NH}{}_n^0}
\newcommand{\la}{\Lambda}
\newcommand{\pak}{\partial_{(k)}}
\newcommand{\cf}{\mathcal{F}}
\newcommand{\cfk}{\mathcal{F}^{(k)}}
\newcommand{\wtfk}{\widetilde{\mathcal{F}}^{(k)}}
\newcommand{\cg}{\mathcal{G}}
\newcommand{\cgf}{\mathcal{G} \otimes \mathcal{F}}
\newcommand{\cfg}{\mathcal{F} \otimes \mathcal{G}}
\newcommand{\wtv}{\widetilde{v}}
\newcommand{\wtu}{\widetilde{u}}
\newcommand{\wte}{\widetilde{e}}

\begin{document}
\title{Towards functor exponentiation}

\author{Mikhail Khovanov and Yin Tian}
\address{Department of Mathematics, Columbia University, New York, NY 10027}
\email{khovanov@math.columbia.edu}
\address{Yau Mathematical Sciences Center, Tsinghua University, Beijing 100084, China}
\email{ytian@math.tsinghua.edu.cn}

\date{December 6, 2017}
\subjclass[2010]{18D10, 17B37, 16E20.}

\keywords{}

\begin{abstract}
We consider a possible framework to categorify the exponential
map exp(-f) given the categorification of a generator f of $\frak{sl}_2$ by Lauda. In this setup
the Taylor expansions of exp(-f) and exp(f) turn into
complexes built out of categorified divided powers of f.
Hom spaces between tensor powers of categorified f are given
by diagrammatics combining nilHecke algebra relations with
those for a additional "short strand" generator. The proposed
framework is only an approximation to
categorification of exponentiation, because the functors categorifying exp(f) and exp(-f) are not invertible.
\end{abstract}

\maketitle

\section{Introduction}

The exponential function is fundamental in mathematics.
In Lie theory, the exponential map connects a Lie algebra and its Lie group.
Idempotented version of quantized universal enveloping algebras of simple Lie algebras have been categorified \cite{KL1,KL2,R}.
This paper can be viewed as a small step towards lifting the exponential map to the categorical level.

We focus on the case of $\frak{sl}_2$.
Consider the expansion
$$\exp(-f)=\sum\limits_{k\ge0}(-1)^k\frac{f^k}{k!},$$
in a completion of the universal enveloping algebra of $\frak{sl}_2$, where $f \in \frak{sl}_2$ is a Chevalley generator of the lower-triangular matrix.
Categorification of the divided power $f^{(k)}=\frac{f^k}{k!}$
and of its quantized version $\frac{f^k}{[k]!}$ naturally appears in the categorified quantum $\frak{sl}_2$ \cite{La}.
The generator $f$ is lifted to a bimodule $\cal{F}'$ over a direct sum of the nilHecke algebras $\bigoplus\limits_{n\ge0} NH_n$.
The tensor powers $\cal{F}'^k=\cal{F}'^{\otimes k}$ of the bimodule admit a direct sum decomposition $\cal{F}'^{k} \cong \bigoplus\limits_{k!} \cal{F}'^{(k)}$.
It is natural to expect lifting $\exp(-f)$ to a cochain complex whose degree $k$ component is $\cal{F}'^{(k)}$ for $k \ge 0$.
A nontrivial differential is needed to link adjacent components.

The diagrammatic approach is widely used in categorification and plays a significant role in the present paper as well.
We provide a modification $\wt{NH}$ of $\bigoplus\limits_{n\ge0} NH_n$ by adding an extra
generator to the nilHecke algebras
The new generator is described by a short strand which links $NH_n$ and $NH_{n+1}$ together.
The induction $\wt{NH}$-bimodule still exists, denoted $\cal{F}$.
Short strand induces a $\wt{NH}$-bimodule homomorphism $\wt{NH} \ra \cal{F}$.
This morphism and its suitable generalizations $\cal{F}^{k} \ra \cal{F}^{k+1}$ define a differential on $ \bigoplus\limits_{k\ge0}\cal{F}^{(k)}[-k] $.
The resulting complex descends to an alternating sum $\sum\limits_{k\ge0}(-1)^k[\cal{F}^{(k)}]$ in the Grothendieck ring of the derived category of $\wt{NH}$-bimodules.

Due to the existence of the short strand, the extension group $\Ext^1(\cal{F}, \wt{NH})$ of bimodules is nontrivial.
We lift the expansion $\exp(f)=\sum\limits_{k\ge0}f^{(k)}$ to a complex $\left( \bigoplus\limits_{k\ge0}\cal{F}^{(k)}, d \right)$, where the differential $d$ consists of certain elements of $\Ext^1$-groups.
The absence of the alternating sign in the expansion of $\exp(f)$ comes from the fact that it cancels against the sign
coming from the use of $\Ext^1$-groups.

Unfortunately, the two resulting complexes, lifting
$\exp(-f)$ and $\exp(f)$, respectively, are not invertible, as explained  in Section \ref{Sec dis}. A more
elaborate or just a different construction is needed to
more adequately categorify exponentiation.


\begin{prob} Find a framework for categorification of the exponential map, where an object $\cal{F}$ in a
monoidal triangulated category $\cal{C}$ lifts to
two invertible objects $\exp(\cal{F})$ and $\exp(-\cal{F})$ in some monoidal triangulated category $\cal{C}^e$.
The objects should descend to $\exp([\cal{F}])$ and $\exp(-[\cal{F}])$ in the Grothendieck ring of $\cal{C}^e$, where
$[\cal{F}]$ is the class of $\cal{F}$ in the Grothendieck ring of $\cal{C}$.
The Grothendieck rings of $\cal{C}$ and $\cal{C}^e$ should be suitably related.
\end{prob}

The recently developed theory of stability in representations of
the symmetric group (see \cite{CEF}, for instance) can be related to a diagrammatical category
similar to the one for $\wt{NH}$. The nilHecke algebra $NH_n$
should be replaced by the group algebra $\mathbb{C}[S_n]$ of the symmetric
group. Adding a short strand with suitable sliding and commutativity
relations will enlarge the direct sum $\oplus_{n\ge 0}\mathbb{C}[S_n]$ to
a non-unital idempotented algebra $\wt{S}$. Alternatively,
this idempotented algebra can be viewed as describing a monoidal
category with a single generating object. One of the first key results
in the stable representation theory of $S_n$ can be restated as
the theorem that the category of finitely-generated right $\wt{S}$
modules is Noetherian.

The first author used endofunctors in the category of finite-dimensional $\oplus_{n\ge 0}\mathbb{C}[S_n]$ modules to categorify the Heisenberg algebra \cite{Kh}.
The second author studied the category of finite dimensional left $\wt{S}$ modules to build a categorical boson-fermion correspondence \cite{T}.
Short strands were also used in diagrammatic categorifications of the polynomial ring in \cite{KS}, and the ring of integers localized at two in \cite{KT}.

One possible application of categorified exponentiation would be the categorification of integral forms of Lie groups and the exponential map between a Lie algebra and its Lie group.
It might also be useful for categorification of Vassiliev invariants, where parameter
$h$ appears as the logarithm of $q$. After categorification $q$ becomes the grading shift, and some sophisticated structure refining the shift functor would
be needed to define its logarithm.

\vspace{0.1in}

\n{\bf Acknowledgments.} M.K. was partially supported by the NSF grants DMS-1406065, DMS-1664255 while working on the paper. Y.T. was partially supported by the NSFC 11601256.

\section{The algebra}

\subsection{The definition of $\nh$}
\begin{defn} \label{def nh}
Define an algebra $\nh$ by generators $1_n$ for $n \ge 0$, $\pa_{i,n}$ for $1 \le i \le n-1$, $x_{i,n}$ for $1 \le i \le n$, and $v_{i,n}$ for $1 \le i \le n$, subject to the relations consisting of three groups:

\n(1) Idempotent relations:
\begin{align*}
1_n 1_m=\delta_{n,m}1_n, &\qquad 1_n \pa_{i,n} = \pa_{i,n} 1_n = \pa_{i,n}, \\
1_n x_{i,n} = x_{i,n} 1_n = x_{i,n}, &\qquad 1_{n-1} v_{i,n} = v_{i,n} 1_n = v_{i,n}.
\end{align*}

\n(2) NilHecke relations:
\begin{align*}
x_{i,n} x_{j,n} = x_{j,n} x_{i,n}, & \\
x_{i,n} \pa_{j,n} = \pa_{j,n} x_{i,n} ~~~\mbox{if}~|i-j|>1, &\qquad  \pa_{i,n} \pa_{j,n} = \pa_{j,n} \pa_{i,n} ~\mbox{if}~~~|i-j|>1, \\
\pa_{i,n} \pa_{i,n} =0, &\qquad \pa_{i,n}\pa_{i+1,n}\pa_{i,n}=\pa_{i+1,n}\pa_{i,n}\pa_{i+1,n}\\
x_{i,n}\pa_{i,n}-\pa_{i,n}x_{i+1,n}=1_n, &\qquad \pa_{i,n}x_{i,n}-x_{i+1,n}\pa_{i,n}=1_n.
\end{align*}

\n(3) Short strands relations:
\begin{align*}
v_{i,n} x_{j,n} = x_{j,n-1} v_{i,n} ~~~\mbox{if}~i > j, &\qquad  v_{i,n} x_{j,n} = x_{j-1,n-1} v_{i,n} ~~~\mbox{if}~i <j, \\
v_{i,n} \pa_{j,n} = \pa_{j,n-1} v_{i,n}, ~~~\mbox{if}~i > j+1, &\qquad  v_{i,n} \pa_{j,n} = \pa_{j-1,n-1} v_{i,n} ~~~\mbox{if}~i <j,  \\
v_{i,n} v_{j,n+1}=v_{j,n} v_{i+1,n+1}~~~\mbox{if}~i \ge j. & \\
v_{i,n} \pa_{i,n} = v_{i+1,n} \pa_{i,n}. & \qquad \mbox{(Exchange relation)}&
\end{align*}
\end{defn}

\begin{figure}[h]
\begin{overpic}
[scale=0.4]{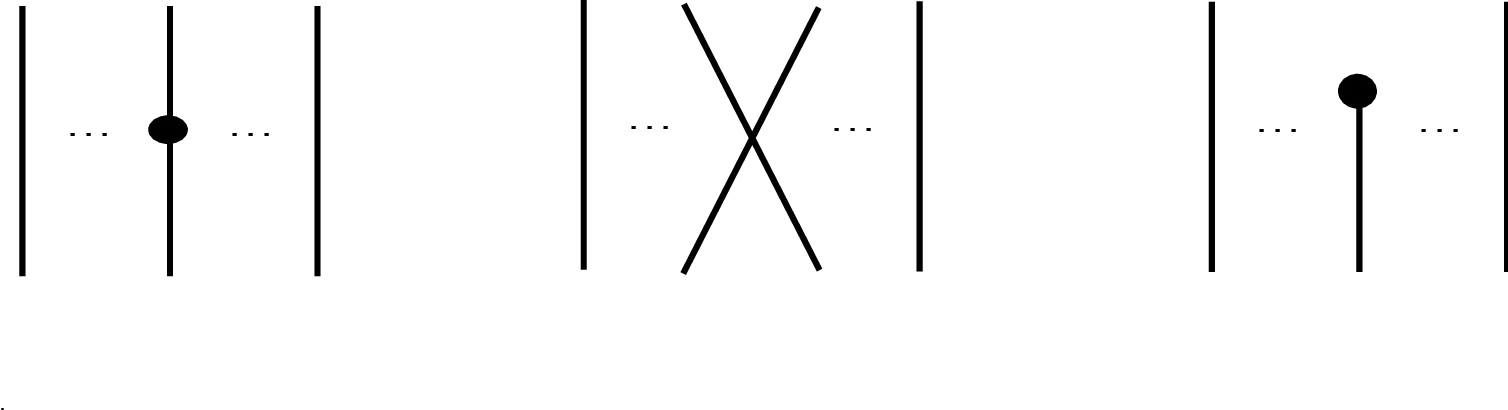}
\put(11,5){$i$}
\put(10,0){$x_{i,n}$}
\put(45,5){$i$}
\put(51,5){$i+1$}
\put(48,0){$\pa_{i,n}$}
\put(89,5){$i$}
\put(88,0){$v_{i,n}$}
\end{overpic}
\caption{Generators of $\nh$.}
\label{fig alg1}
\end{figure}

The algebra $\nh$ can be described diagrammatically.
The idempotent $1_n$ is denoted by $n$ vertical strands.
In particular, $1_0$ is denoted by the empty diagram.
The generator $x_{i,n}$ is denoted by $n$ vertical strands with a dot on the $i$-th strand, and $\pa_{i,n}$ is denoted by $n$ strands with a $(i,i+1)$ crossing.
The new generator $v_{i,n}$ is denoted by a diagram with $n-1$ vertical strands and one short strand in the $i$-th position.
The short strand has no endpoint at the top and one endpoint at the bottom, see figure \ref{fig alg1}.

The product $ab$ of two diagrams $a$ and $b$ is a vertical concatenation of $a$ and $b$, where $a$ is at the top, $b$ is at the bottom.
The product is zero unless the numbers of their endpoints match.

The relations of the second group are the defining relations of the nilHecke algebras.
The relations of the third group are about short strands.
The first three lines are isotopy relations of disjoint diagrams.
The last line says that the short strand is exchangeable between $i$-th and $(i+1)$-th positions when composing with the crossing.
We call it the {\em exchange relation}.
In addition to the isotopy relations of disjoint diagrams, other local relations are drawn in figure \ref{fig alg2}.

Let $a \od b \in \nh$ denote a horizontal concatenation of $a$ and $b$, where $a$ is on the left, $b$ is on the right.
The element $a \od b$ does not depend on the heights of $a$ and $b$ by the isotopy relations of disjoint diagrams.
\begin{figure}[h]
\begin{overpic}
[scale=0.3]{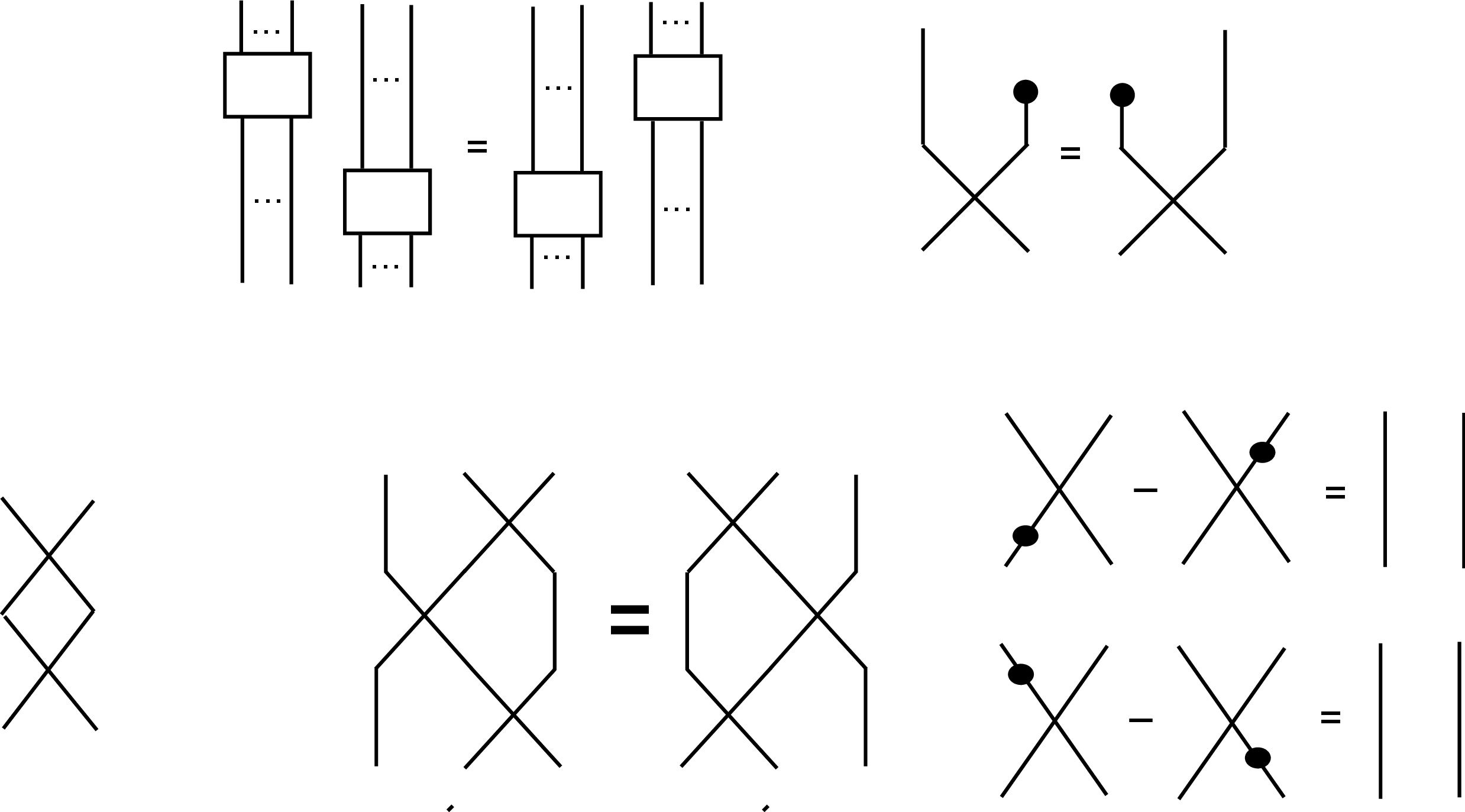}
\put(30,30){$a \od b$}
\put(18,49){$a$}
\put(45,49){$b$}
\put(37,41){$a$}
\put(26,41){$b$}
\put(9,13){$=0$}
\end{overpic}
\caption{Local relations of $\nh$.}
\label{fig alg2}
\end{figure}

The algebra $\nh$ is idempotented, i.e. has a complete system of mutually orthogonal idempotents $\{1_n\}_{n\ge 0},$ so that
$$ \nh \ = \ \bigoplus\limits_{m,n\ge 0} 1_m \nh 1_n.$$
Let $\nhmn$ denote its component $1_m \nh 1_n$.
It is spanned by diagrams with $m$ endpoints at the top and $n$ endpoints at the bottom.
The new generator $v_{i,n} \in \nh^{n-1}_{n}$.
Clearly, $\nhmn=0$ if $n<m$.
Let $NH_n$ denote the nilHecke algebra with the generators $x_i$ and $\pa_i$.
There is a surjective homomorphism $\psi_n: NH_n \ra \nh_{n}^{n}$ defined by $\psi_n(1)=1_n, \psi_n(x_i)=x_{i,n}, \psi_n(\pa_{i})=\pa_{i,n}$.
We will prove that $\psi_n$ is an isomorphism in Section \ref{Sec basis}.

A combination of the exchange relation, the isotopy relation and the nilHecke relations implies the following relations, see figure \ref{fig alg3}.
\begin{figure}[h]
\begin{overpic}
[scale=0.3]{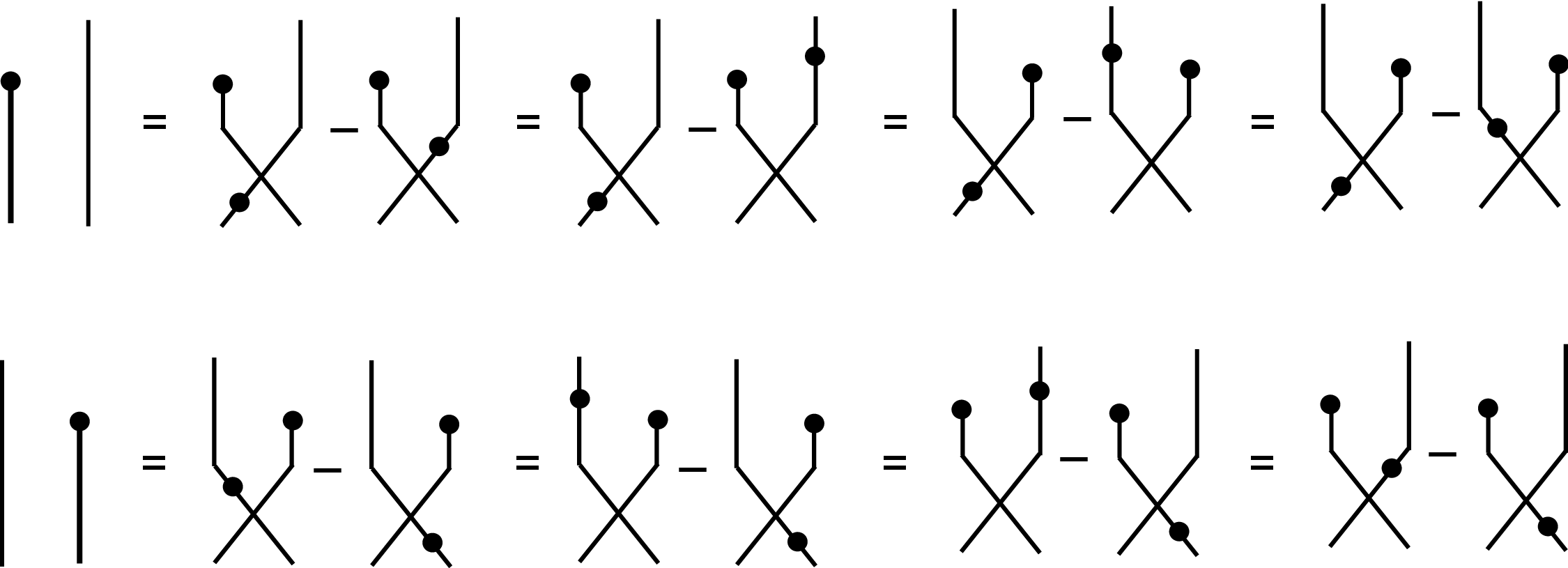}
\end{overpic}
\caption{Two induced relations}
\label{fig alg3}
\end{figure}

Motivated by these relations, we introduce a new diagram as a circle crossing, see figure \ref{fig alg4}.
It defines an element
\begin{gather} \label{def s}
s_{1,2}:=\pa_{1,2} x_{1,2} - x_{1,2} \pa_{1,2} = x_{2,2} \pa_{1,2} - \pa_{1,2} x_{2,2}  \in NH_2^2.
\end{gather}
The relations in figure \ref{fig alg3} can be rewritten in terms of the new diagram $s_{1,2}$.
The short strand can slide through the circle crossing.
We call this relation the {\em slide relation}.
As a corollary, the horizontal concatenation of two short strands is fixed by the circle crossing.
We call it the {\em relation (F)}.
\begin{figure}[h]
\begin{overpic}
[scale=0.3]{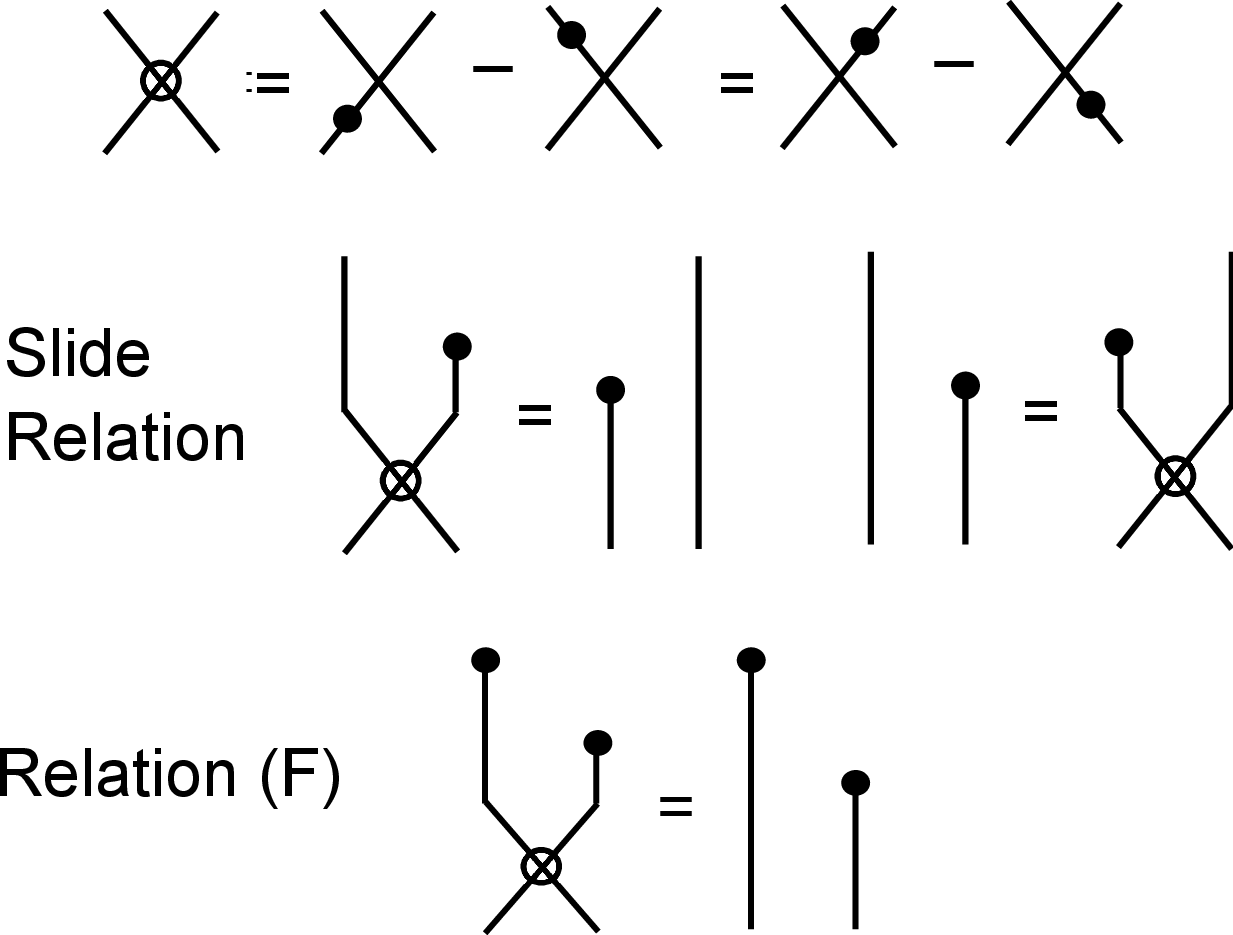}
\end{overpic}
\caption{A new diagram $s_{1,2}$, the slide relation, and the relation (F).}
\label{fig alg4}
\end{figure}

We discuss the interaction of the new diagram $s_{1,2}$ with the generators of $\nh$ in the following.
Firstly, we check that $s_{1,2}s_{1,2}=1_2$, and
\begin{gather} \label{eq s cr}
s_{1,2} \pa_{1,2} = \pa_{1,2}, \qquad  \pa_{1,2} s_{1,2} = - \pa_{1,2}.
\end{gather}
using diagrams, see figure \ref{fig alg5}.
\begin{figure}[h]
\begin{overpic}
[scale=0.3]{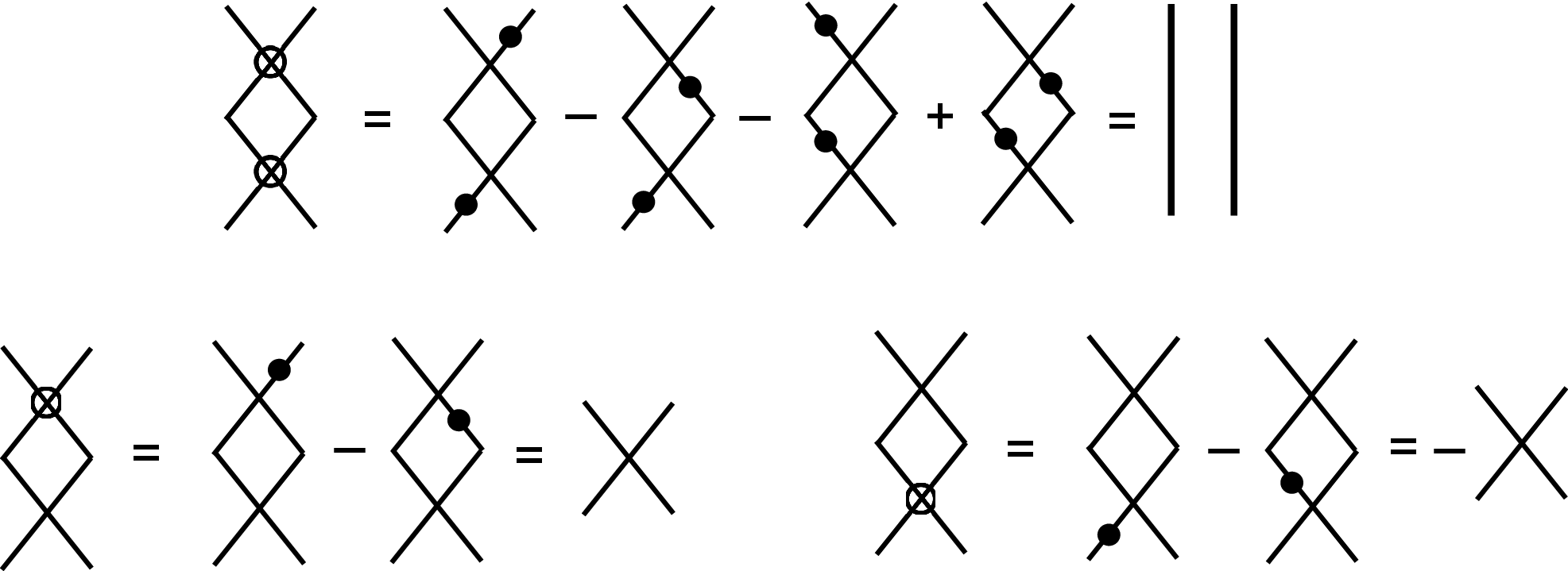}
\end{overpic}
\caption{Local relations about the circle crossing.}
\label{fig alg5}
\end{figure}
Secondly, we use the nilHecke relations to deduce more relations, see figure \ref{fig alg6}.
For $1\le i \le n-1$, define
$$s_{i,n}=1_{i-1} \od s_{1,2} \od 1_{n-i-1} \in NH_n^n$$
by adding $i-1$ and $n-i-1$ vertical strands to the left and right of $s_{1,2}$, respectively.
\begin{figure}[h]
\begin{overpic}
[scale=0.3]{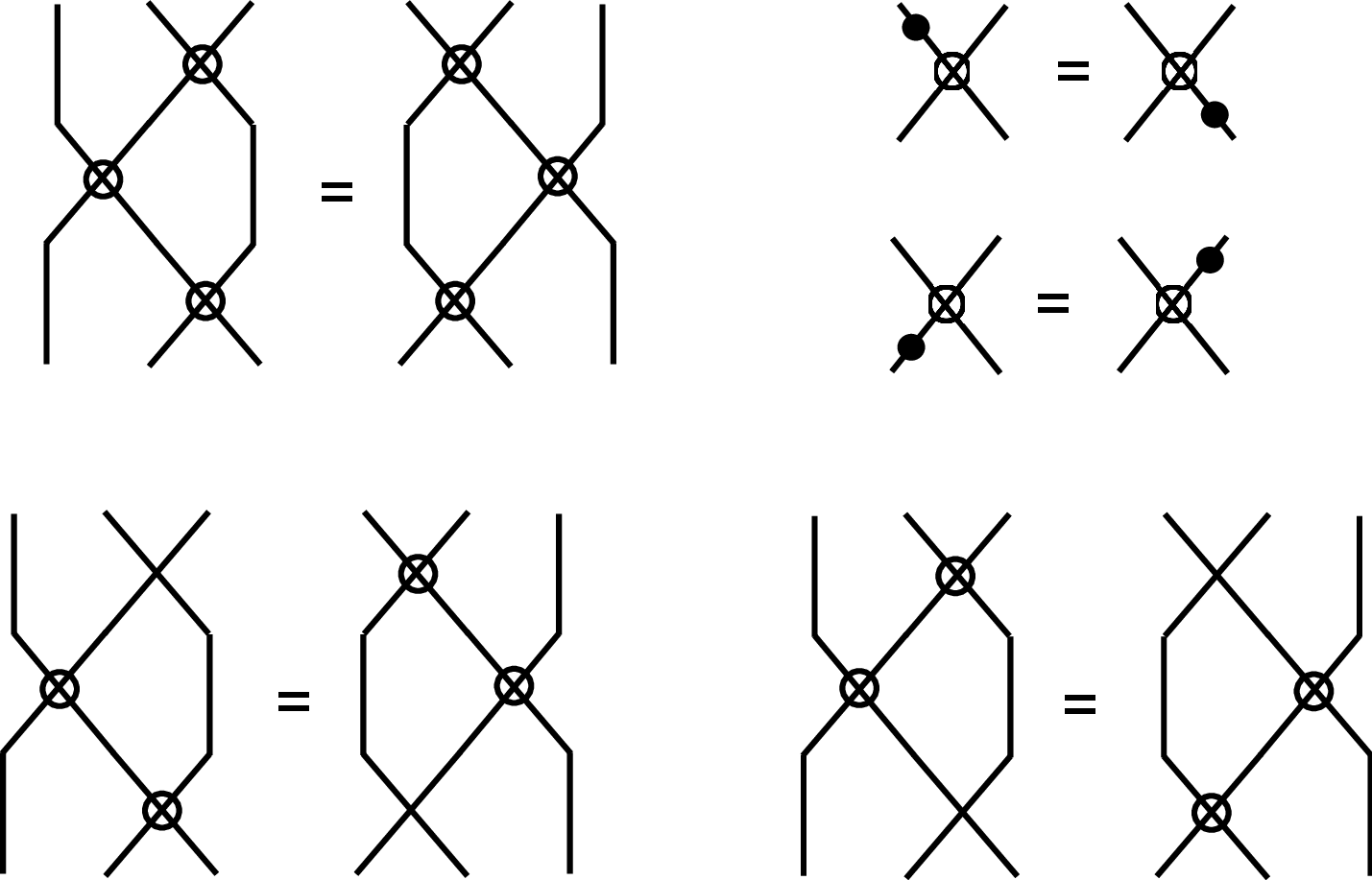}
\end{overpic}
\caption{More local relations for the circle crossings.}
\label{fig alg6}
\end{figure}

The diagrams used in defining $s_{1,2}$ do not have the short strand.
The same diagram defines an element $s_1:=\pa_{1} x_{1} - x_{1} \pa_{1} \in NH_2$ such that $\psi_2(s_1)=s_{1,2}$.
Similarly, define $s_i:=\pa_{i} x_{i} - x_{i} \pa_{i} \in NH_n$ such that $\psi_n(s_i)=s_{i,n}$.

\begin{lemma} \label{lem s}
(1) For $1\le i \le n-1$, the generators $s_{i,n}$ satisfy the defining relations of the symmetric group $S_n$.

\n(2) A dot can slide through the circle crossing, and a crossing and two circle crossings satisfy the braid relation, see figure \ref{fig alg6}.
\end{lemma}
\begin{proof}
Under the map $\psi_n: NH_n \ra \nh_n^n$, it is enough to check the corresponding relations in $NH_n$.
Consider the action of $NH_n$ on the ring $\K[x_1, \dots, x_n]$ of polynomials, where $\pa_{i}$ acts as the divided difference operator, and $x_{i}$ acts as the multiplication by $x_i$.
For $f \in \K[x_1, \dots, x_n]$,
\begin{align*}
s_{i}(f) & = \pa_{i} x_{i}(f) - x_{i} \pa_{i} (f) = \pa_i (x_i f) - x_i \pa_i(f) \\
& = \frac{x_i f - x_{i+1} \wt{s}_i(f)}{x_i - x_{i+1}} - x_i \cdot \frac{f -  \wt{s}_i(f)}{x_i - x_{i+1}} = \wt{s}_i(f),
\end{align*}
where, $\wt{s}_i(f)(\dots, x_i, x_{i+1}, \dots)=f(\dots, x_{i+1}, x_{i}, \dots)$.
In other words, the operator induced by $s_{i}$ is the same as the operator $\wt{s}_i$ on $\K[x_1, \dots, x_n]$.
The relations hold in $NH_n$ since the action is faithful.
\end{proof}

The slide relation in figure \ref{fig alg4} says that the short strand can slide through the circle crossing.
Thus, we can reduce the generators of $\nh$ in Definition \ref{def nh} to $1_n, \pa_{i,n}, x_{i,n}$, and $v_{n,n}$.
Here, $v_{n,n}$ is the diagram with the short strand in the rightmost position.

\vspace{.1cm}
\n{\bf Notation: }Let $v_n=v_{n,n} \in \nh^{n-1}_{n}$ for simplicity.

\begin{defn} \label{def nh'}
Define a $\K$-algebra $\nh'$ by generators $1_n$ for $n \ge 0$, $\pa_{i,n}$ for $1 \le i \le n-1$, $x_{i,n}$ for $1 \le i \le n$, and $v_{n}$ for $n \ge 0$, subject to the relations consisting of three groups:

\n(1) Idempotent relations: the same as in Definition \ref{def nh} (1), except that $v_{i,n}$ is replaced by $v_n$.

\n(2) NilHecke relations: the same as in Definition \ref{def nh} (2).

\n(3) Short strands relations:
\begin{align*}
v_{n} x_{j,n} = x_{j,n-1} v_{n}, & \qquad \mbox{if}~j<n, \\
v_{n} \pa_{j,n} = \pa_{j,n-1} v_{n}, & \qquad\mbox{if}~j<n-1,   \\
v_{n-1} v_{n}=v_{n-1} v_{n} s_{n-1,n}, & \qquad \mbox{(Relation (F))}
\end{align*}
where $s_{n-1,n}=\pa_{n-1,n} x_{n-1,n} - x_{n-1,n} \pa_{n-1,n} = x_{n,n} \pa_{n-1,n} - \pa_{n-1,n} x_{n,n}$.
\end{defn}

\begin{lemma} \label{lem nh nh'}
The algebras $\nh$ and $\nh'$ are naturally isomorphic.
\end{lemma}
\begin{proof}
The algebra $\nh'$ has fewer generators and relations than $\nh$, except for the last relation in $\nh'$.
This relation is a direct consequence of the relation (F) in figure \ref{fig alg4}.
Therefore, there is a homomorphism $\nh' \ra \nh$.

We construct a map $\phi: \nh \ra \nh'$ in the opposite direction as follows.
On generators, define $\phi(a)=a$ for $a=1_n, \pa_{i,n}, x_{i,n}$, and
$$\phi(v_{n,n})=v_{n}, \qquad \phi(v_{i,n})=v_{n} s_{n-1,n} \cdots s_{i+1,n} s_{i,n}.$$
The definition of $\phi(v_{i,n})$ is motivated by the slide relation.
We need to prove that $\phi$ respects the defining relations of $\nh$.
This is clear for the relations (1), (2) in Definition \ref{def nh}.
Consider the short strand relations, see Definition \ref{def nh} (3).
In figure \ref{fig alg7}, we check some special cases using the relations in figures \ref{fig alg4}, \ref{fig alg5} and \ref{fig alg6}.
For $v_{i,n} x_{j,n} = x_{j-1,n-1} v_{i,n}$ when $i <j$, we check that
\begin{align*}
\phi(v_{i,n})\phi(x_{j,n})&=v_{n} s_{n-1,n} \cdots s_{i+1,n} s_{i,n}x_{j,n}=v_{n}x_{j-1,n}s_{n-1,n} \cdots s_{i+1,n} s_{i,n} \\
&=x_{j-1,n-1}v_{n} s_{n-1,n} \cdots s_{i+1,n} s_{i,n}=\phi(x_{j-1,n-1})\phi(v_{i,n}).
\end{align*}
The proof for the general cases is similar and left to the reader.
\begin{figure}[h]
\begin{overpic}
[scale=0.25]{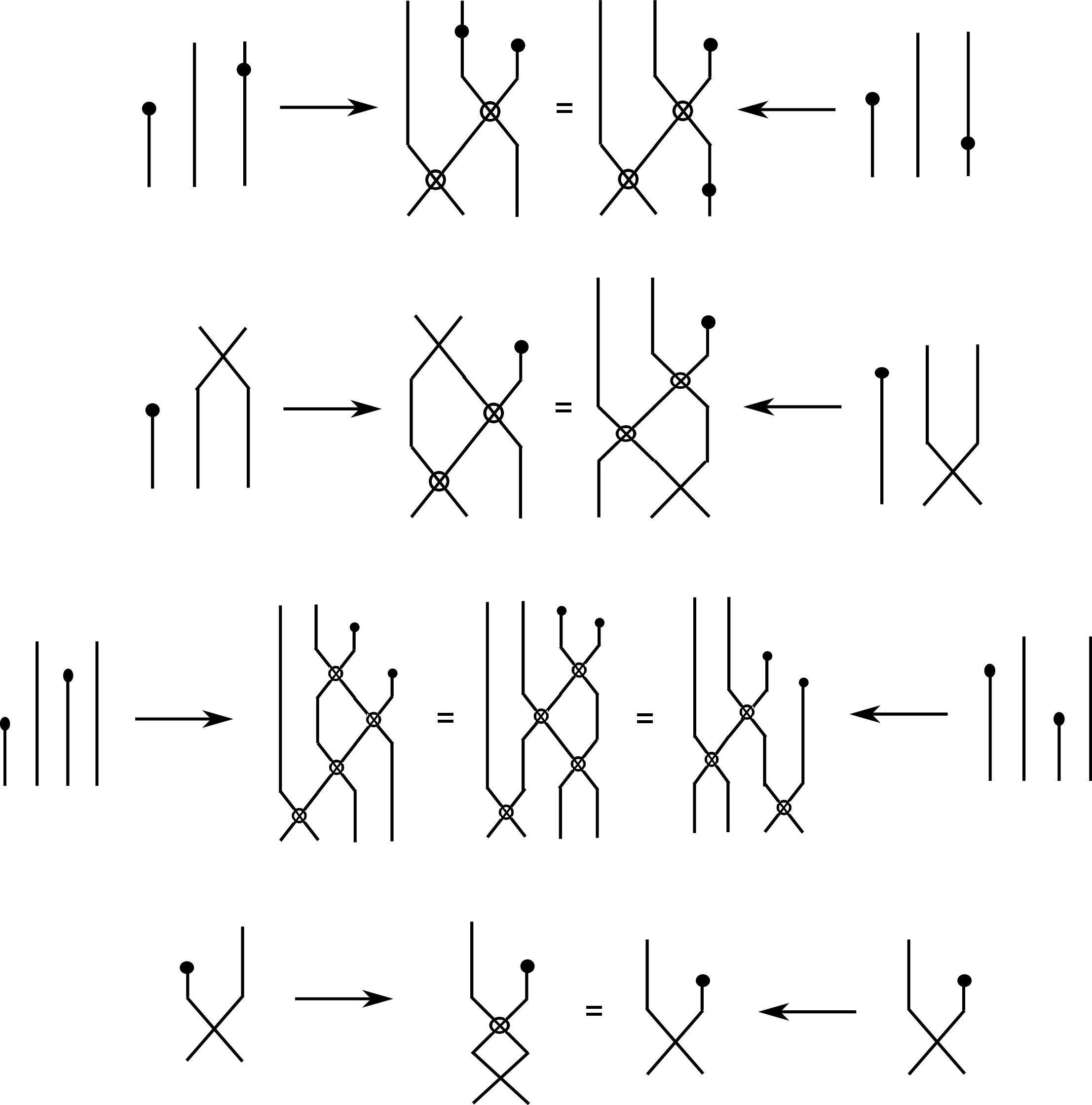}
\end{overpic}
\caption{Checking $\phi$ preserves the short strand relations, where the diagrams on the left and right sides are in $\nh$, the arrows represent the map $\phi$, and the diagrams in the middle between the arrows are in $\nh'$.}
\label{fig alg7}
\end{figure}

It is obvious that the two maps between $\nh$ and $\nh'$ are inverse to each other.
\end{proof}

We will identify $\nh'$ with $\nh$ under the isomorphism in the lemma.
Definition \ref{def nh'} has fewer generators and relations, while Definition \ref{def nh} is more symmetric and local diagrammatically.
Moreover, it is clear that Definition \ref{def nh} describes Hom spaces between powers of a generating object of a suitable monoidal category.
We will use both presentations.

\subsection{The basis} \label{Sec basis}
We first construct an action of $\nh$ which generalizes the action of nilHecke algebras on the rings of polynomials.
Let
\begin{gather} \label{def R}
R=\Z[r_1, r_2, \dots]
\end{gather}
be a graded ring of polynomials in infinitely many variables, where $\deg(r_i)=2i$.
Let $R_n=R[x_1, \dots, x_n]$ be the graded ring of polynomials over $R$, where $\deg(x_i)=2$.

In the following, we define an action $\alpha$ of $\nh=\bigoplus\limits_{m,n \ge 0} \nh^m_n$ on $\wt{R}=\bigoplus\limits_{n \ge 0} R_n$, where on each component $\alpha: \nh^m_n \ra \Hom_R(R_n, R_m)$.
For $m=n$, the action is the same as the nilHecke action on $R_n$:
$$\alpha(x_{i,n})(f)=x_i f, \qquad \alpha(\pa_{i,n})(f)=\frac{f-s_i(f)}{x_i-x_{i+1}},$$
for $f \in R_n$.
For $v_n \in \nh^{n-1}_n$, the operator $\alpha(v_n): R_n \ra R_{n-1}$ is defined on the $R$-basis $\{x_1^{i_1}\dots x_{n-1}^{i_{n-1}} x_{n}^{i_n}\}$ of $R_n$ by
\begin{gather} \label{eq action}
\alpha(v_n)(x_1^{i_1}\dots x_{n-1}^{i_{n-1}} x_{n}^{i_n}) = x_1^{i_1}\dots x_{n-1}^{i_{n-1}} r_{i_n} \in R_{n-1}.
\end{gather}
In other words, $\alpha(v_n)$ fixes the first $n-1$ variables, and maps $x_n^k$ to $r_k$.

To check that the action is well-defined, we use the presentation of $\nh$ in Definition \ref{def nh'}.
Since $\alpha(v_n)$ acts nontrivially only on the variable $x_n$, all isotopy relations and nilHecke relations are preserved under $\alpha$.
For the relation (F), both sides map $x_1^{i_1}\dots x_{n-2}^{i_{n-2}} x_{n-1}^{i_{n-1}} x_{n}^{i_n}$ to $x_1^{i_1}\dots x_{n-2}^{i_{n-2}} r_{i_{n-1}} r_{i_n}$.

\vspace{.2cm}
To compute the basis of $\nhmn$, we first discuss the case $m=0$.
Recall from Lemma \ref{lem s} that the $s_{i,n}$'s generate a symmetric group $S_n \subset NH_n$.
Let $$V_n=\Ind_{\Z[S_n]}^{NH_n} \Z$$ be the right $NH_n$-module induced from the trivial right $\Z[S_n]$-module $\Z$.
Let $NH_n \ra V_n$ denote the quotient map which takes $t \in NH_n$ to $\ov{t} \in V_n$.
As an abelian group, $V_n$ is generated by elements $\ov{t}$ for $t \in NH_n$, modulo the relation $\ov{w t}=\ov{t}$ for $w \in S_n$.

Let $\vzn=v_1v_2\cdots v_n \in \nhzn$ denote the diagram of $n$ short strands.
Define a map of abelian groups:
\begin{align} \label{eq eta}
\begin{array}{cccc}
\eta_n: & V_n & \ra & \nhzn \\
& \ov{t} & \mapsto & \vzn t.
\end{array}
\end{align}
The map $\eta_n$ is well-defined since $\eta_n(\ov{w t})=\vzn w t = \vzn t=\eta_n(\ov{t})$ for $w \in S_n$ from the relation (F), see figure \ref{fig alg4}.
It is surjective by definition.
The map $\eta_n$ is actually a homomorphism of right $NH_n$-modules.
Note that both $V_n$ and $\nhzn$ are graded.
Let $V_n(m)$ and $\nhzn(m)$ be their degree $2m$ components.

We will prove that $V_n$ is a free abelian group and show that $\eta_n$ is an isomorphism in the following.
Recall some results about the action of $NH_n$ on $\cal{P}_n=\Z[x_1, \dots, x_n]$, see \cite[Section 3.2]{La} for more detail.
Let $\la_n\subset \cal{P}_n$ denote the subring of symmetric polynomials.
Let $$\cal{H}_n=\{x_1^{j_1}\dots  x_{n}^{j_n} ~|~ 0 \le j_k \le n-k\},$$
and $\cal{H}_n(m)$ be its subset of elements of degree $2m$.
Then $\cal{P}_n$ is a finite-rank free module over $\la_n$ with basis $\cal{H}_n$.
There is a canonical isomorphism of rings
$$NH_n \cong \End_{\la_n}(\cal{P}_n).$$
Let $\pa_{w}=\pa_{i_1}\cdots \pa_{i_l} \in NH_n$ for a reduced word expression of $w=s_{i_1}\cdots s_{i_l} \in S_n$.
The number $l$ is called the length $l(w)$ of $w$.
Define $$t_{g,w}=g\cdot \pa_{w}\in NH_n,$$ for $w \in S_n, g \in \cal{P}_n$.
The collection $\{t_{g,w}~|~ w \in S_n, g ~\mbox{a monomial in}~ \cal{P}_n\}$ forms a $\Z$-basis of $NH_n$, see \cite[Proposition 3.5]{La}.
Let
$$\cal{F}_n(m)=\{x_1^{i_1}\dots  x_{n}^{i_n} ~|~ i_1  \ge \cdots \ge i_n \ge 0, \sum\limits_k i_k=m\},$$
$$\cal{B}_n(m)=\{t_{f,w} ~|~ w \in S_n, f \in \cal{F}_n(l(w)+m)\}.$$
Here, $\cal{F}_n(m)$ is a finite subset of $\cal{P}_n$, and $\cal{B}_n(m)$ a finite subset of $NH_n(m)$.
Moreover, $\cal{B}_n(m)$ is empty if $m<-\binom{n}{2}$.

\begin{lemma} \label{lem vn}
The abelian group $V_n(m)$ is free with a basis which is in bijection with $\cal{B}_n(m)$.
\end{lemma}
\begin{proof}
Firstly, we claim that $\{\ov{t}_{f,w}~|~t_{f,w} \in \cal{B}_n(m)\}$ generates $V_n(m)$ over $\Z$.
Let $g=x_{l_1}^{i_1}\dots  x_{l_n}^{i_n}$ such that $i_1 \ge \cdots \ge i_n \ge 0$, and $\sum\limits_k i_k=l+m$.
Let $w_g \in S_n$ denote the permutation which maps $l_k$ to $k$.
Then $f= w_g (g) = x_1^{i_1}\dots  x_{n}^{i_n} \in \cal{F}_n(l+m)$.
So $w_g \cdot t_{g,w}=t_{f,w} \in NH_n$, and $\ov{t}_{g,w}=\ov{t}_{f,w}$.
The claim follows from that $\{t_{g,w}\}$ generates $NH_n(m)$ over $\Z$.

Secondly, we claim that the collection $\{\ov{t}_{f,w}~|~ t_{f,w} \in \cal{B}_n(m)\}$ is $\Z$-independent in $V_n(m)$.
Let $$\beta_n=\alpha \circ \eta_n: V_n \ra \nhzn \ra \Hom_R(R_n ,R).$$
Suppose that $\gamma=\sum k_{f,w}~ \beta_n(\ov{t}_{f,w})=\sum k_{f,w}~ \alpha(v_n^0 \cdot t_{f,w})=0 \in \Hom_R(R_n ,R)$ for some integers $k_{f,w}$.
We have $0=\gamma(1)=\sum k_{f,1}~ \alpha(v_n^0 \cdot t_{f,1})(1)$ since $\pa_w(1)=0$ for $w \neq 1 \in S_n$.
The image $\alpha(v_n^0 \cdot t_{f,1})(1)=r_{i_1}\cdots r_{i_n}$ for $f=x_1^{i_1}\cdots  x_{n}^{i_n} \in \cal{F}_n$.
Since $\{r_{i_1}\cdots r_{i_n}~|~i_1  \ge \cdots \ge i_n \ge 0\}$ is $\Z$-independent in $R$, we have $k_{f,1}=0$ for all $f \in \cal{F}_n$.
For $w_1 \in S_n$ with $l(w_1)=1$, let $h_{w_1} \in \cal{P}_n$ be the Schubert polynomial associated to $w_1$.
Then $\pa_{w_1}(h_{w_1})=1$, and $\pa_{w}(h_{w_1})=0$ for $w \neq w_1$ or $1$.
We have $0=\gamma(h_{w_1})=\sum k_{f,w_1}~ \alpha(v_n^0 \cdot t_{f,w_1})(h_{w_1})$.
It follows that $k_{f,w_1}=0$ for all $f \in \cal{F}_n$.
By applying $\gamma$ to all Schubert polynomials in $\cal{P}_n$, one can inductively show that $k_{f,w}=0$ for all $f \in \cal{F}_n$ and $w \in S_n$.
We conclude that $\{\beta_n(\ov{t}_{f,w})~|~ t_{f,w} \in \cal{B}_n(m)\}$ is $\Z$-independent.
\end{proof}

The map $\beta_n$ in the proof above is injective.
It implies that $\eta_n$ is also injective.
Combining with the fact that $\eta_n$ is surjective, we have the following description of $\nhzn$.

\begin{prop} \label{prop basis nh0n}
The map $\eta_n: V_n \ra  \nhzn$ is an isomorphism of free abelian groups.
Moreover, $\nhzn$ has a $\Z$-basis $\{v_n^0 t~|~ t \in \cal{B}_n(m), m \in \Z\}$ which is in bijection with $\cal{B}_n=\bigcup\limits_m\cal{B}_n(m)$.
\end{prop}

The abelian group $\nhzn$ has a $\Z$-basis $\{v^0_n \cdot f \cdot \pa_{w} ~|~ w \in S_n, f \in \cal{F}_n\}$, where $\cal{F}_n=\bigcup\limits_m\cal{F}_n(m)$.
Diagrammatically, $f$ only has dots and numbers of dots are non-increasing from left to right, and $\pa_w$ only has crossings.  See the picture on the left in figure \ref{fig alg8}.

\begin{figure}[h]
\begin{overpic}
[scale=0.25]{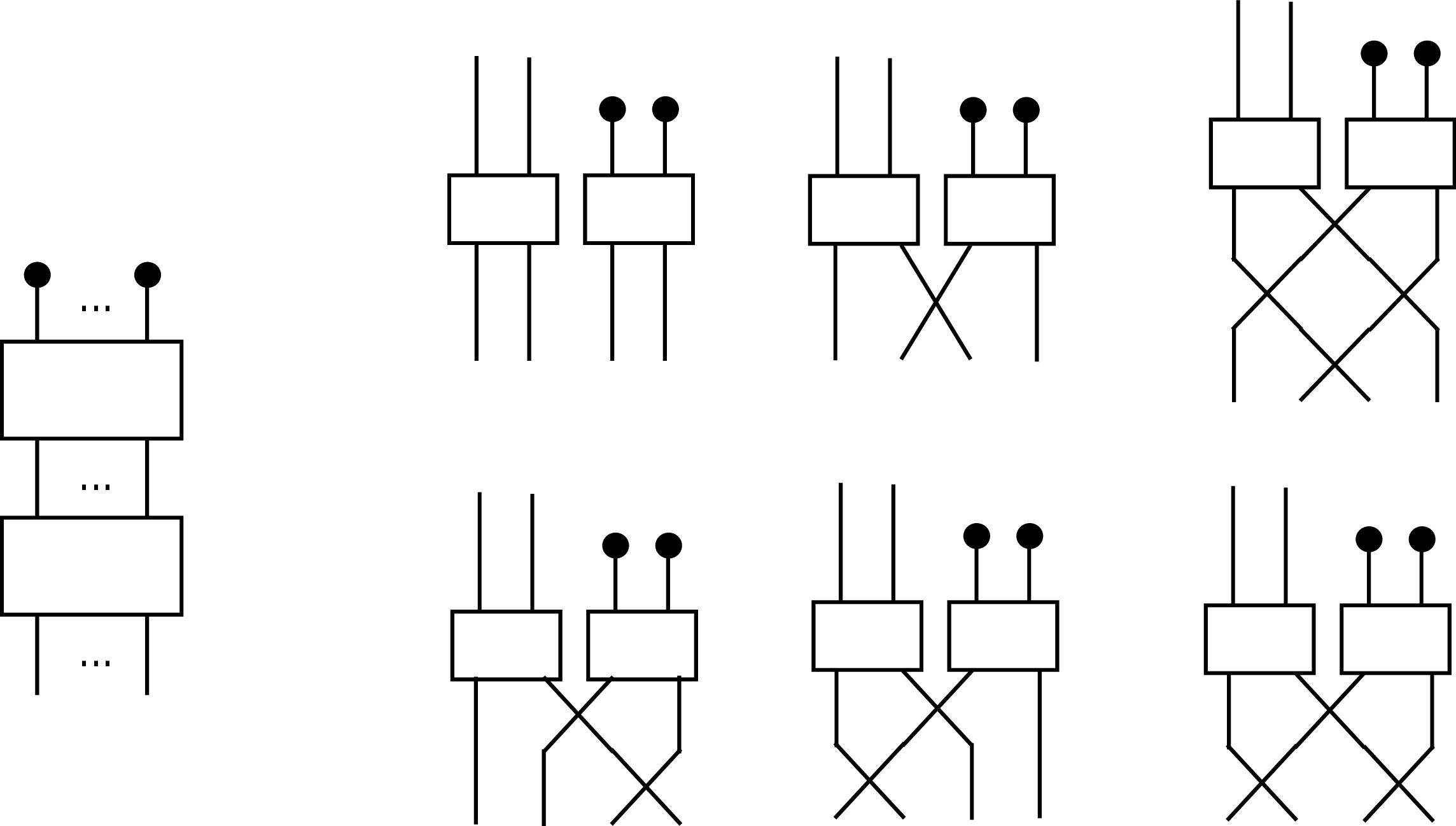}
\put(5,29){$f$}
\put(4,16){$\pa_w$}
\put(31.5,42){$\scriptscriptstyle{NH_k}$}
\put(41,42){$\scriptscriptstyle{NH^0_n}$}
\put(56.5,42){$\scriptscriptstyle{NH_k}$}
\put(65.5,42){$\scriptscriptstyle{NH^0_n}$}
\put(84,45.5){$\scriptscriptstyle{NH_k}$}
\put(93,45.5){$\scriptscriptstyle{NH^0_n}$}
\put(31.5,12){$\scriptscriptstyle{NH_k}$}
\put(41,12){$\scriptscriptstyle{NH^0_n}$}
\put(56.5,12){$\scriptscriptstyle{NH_k}$}
\put(65.5,12){$\scriptscriptstyle{NH^0_n}$}
\put(84,12){$\scriptscriptstyle{NH_k}$}
\put(93,12){$\scriptscriptstyle{NH^0_n}$}
\end{overpic}
\caption{The picture on the left describes the basis of $\nhzn$, where $f \in \cal{F}, \pa_w \in NH_n$; the six pictures on the right describe the basis of $\nh^k_{n+k}$ in terms of bases of $NH_k$ and $\nh^0_n$ in the case $n=k=2$.}
\label{fig alg8}
\end{figure}

The basis of $\nh^k_{n+k}$ can be described using bases of $NH_k$ and $\nhzn$.
Let $S_n \subset NH_{n+k}$ be the symmetric group generated by $s_{i,n+k}$'s for $i=k+1,\dots,n+k-1$.
Let $NH_k \subset NH_{n+k}$ be the inclusion given by adding $k$ vertical strands on the right.
Then the left actions of $S_n$ and $NH_k$ on $NH_{n+k}$ commute.
Define the induction module
$$V^k_{n+k}=\Ind_{\Z[S_n]}^{NH_{n+k}} \Z.$$
It is a left $NH_k$, right $NH_{n+k}$ module.
We call it an $(NH_k, NH_{n+k})$-bimodule.
Since $NH_{n+k}$ is a free left module over $NH_k$, $V^k_{n+k}$ is also free over $NH_k$.
There is a canonical surjective map $NH_{n+k} \ra V^k_{n+k}$.
Define a map
\begin{align} \label{eq eta2}
\begin{array}{cccc}
\eta^k_{n+k}: & V^k_{n+k} & \ra & \nh^k_{n+k} \\
& \ov{t} & \mapsto & v^k_{n+k} t,
\end{array}
\end{align}
for $t \in NH_{n+k}$, where $v^k_{n+k} \in \nh^k_{n+k}$ has $n$ short strands on the right.
It is known that $NH_k$ is a free abelian group and has a basis of the form
$$\cal{B}^k_k=\{f \cdot \pa_w ~|~w \in S_k,~ f ~\mbox{a monomial in}~ x_1,\dots,x_k\}.$$
By a similar argument as in the proof of Proposition \ref{prop basis nh0n}, one can prove the following result.

\begin{prop} \label{prop basis nh}
The map $\eta^k_{n+k}: V^k_{n+k} \ra  \nh^k_{n+k}$ is an isomorphism of $(NH_k, NH_{n+k})$-bimodules.
As an abelian group, $\nh^k_{n+k}$ is free with a basis $$\{v^k_{n+k} \cdot (t_k \od b_n) \cdot \pa_w ~|~ t_k \in \cal{B}_k^k, b_n \in \cal{B}_n, w ~\mbox{a minimal representative in}~ (S_{k} \times S_{n})\backslash S_{n+k}\}.$$
\end{prop}

See figure \ref{fig alg8} for an example with $n=k=2$.

\section{The complex lifting exponentials}

\subsection{The $\nh$-bimodule $\cf$.}

There is an inclusion
\begin{align} \label{eq inclusion}
\begin{array}{cccc}
\rho: & \nh & \ra & \nh \\
& a & \mapsto & a \od 1_1,
\end{array}
\end{align}
given by adding one vertical strand on the right to any diagram $a \in \nh$.
Consider the $\nh$-bimodule corresponding to the induction functor with respect to $\rho$.

\begin{defn} \label{def F}
Define an abelian group $\cf= \bigoplus\limits_{m \ge 0, n \ge 1}\nhmn$ with the $\nh$-bimodule structure given by
$a\cdot t \cdot b=a~t~\rho(b)$, where $a~t~\rho(b)$ is the product in $\nh$, for $a,b \in \nh$ and $t \in \cf$.
\end{defn}

Recall that $\nhmn=0$ if $m>n$.
See figure \ref{fig alg10} for a diagrammatic description of $\cf$.
The left and right multiplication correspond to stacking diagrams at the top and bottom, respectively.
Since the rightmost strand in $\cf$ is unchanged under stacking diagrams from the bottom, we call it the {\em frozen strand} of $\cf$, and add a little bar at its lower end.
\begin{figure}[h]
\begin{overpic}
[scale=0.3]{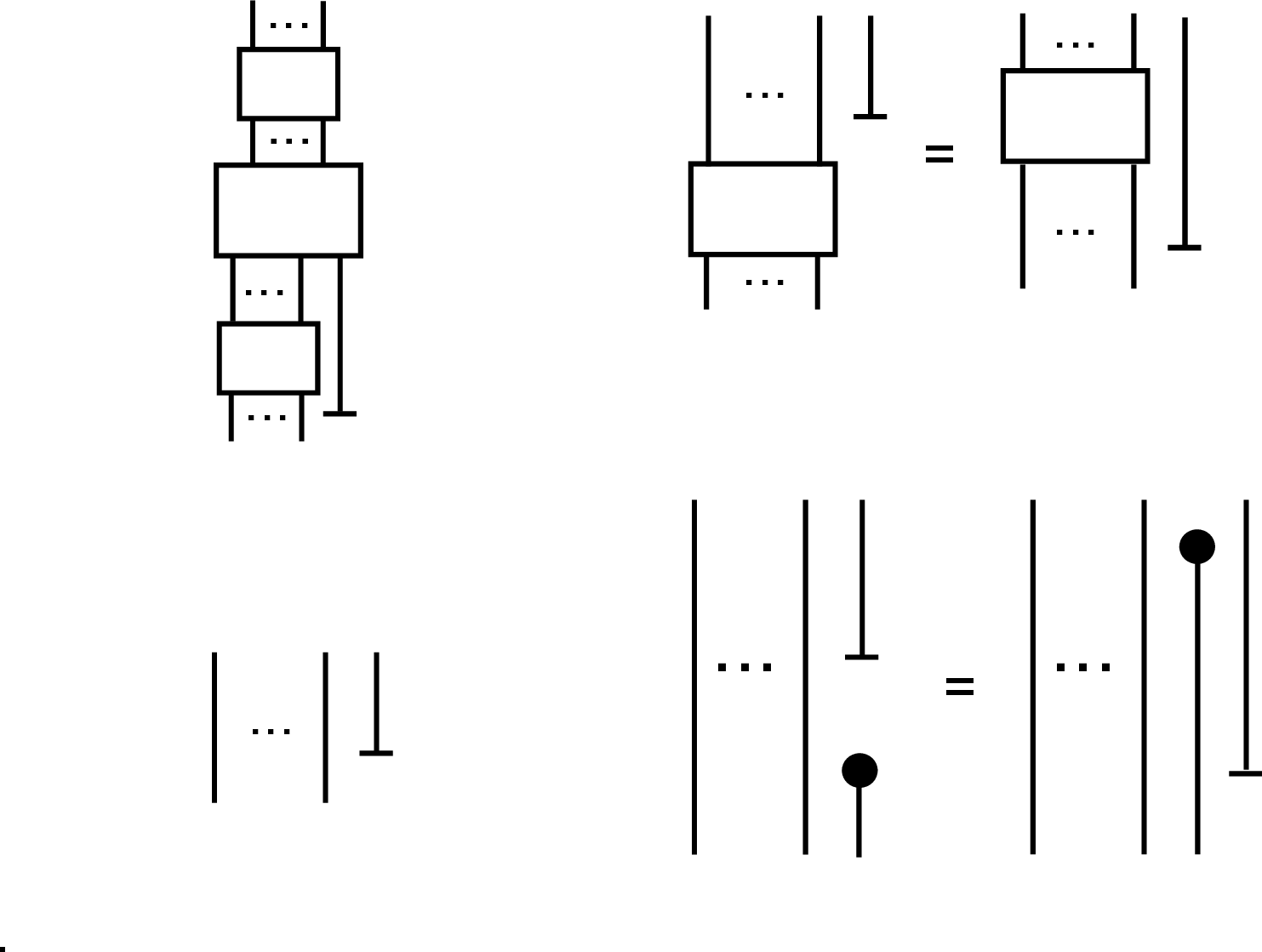}
\put(0,54){$\cf:$}
\put(-2,15){$[1_n]_f:$}
\put(21,25){$n$}
\put(19,8){$\scriptstyle{n-1}$}
\put(21,68){$a$}
\put(21,57){$t$}
\put(20,45){$b$}
\put(60,57){$a$}
\put(84,64){$a$}
\put(50,42){$[1_{n}]_f \cdot a  =[(a \od 1_1)\cdot 1_{n}]_f$}
\put(50,0){$[1_{n}]_f \cdot v_n  =[v_{n,n+1}\cdot 1_{n+1}]_f$}
\end{overpic}
\caption{The bimodule $\cf$, its generators $[1_n]_f$, and the right multiplication on $\cf$.}
\label{fig alg10}
\end{figure}

Let $P_n=\nh \cdot 1_n$ denote the left projective $\nh$-module associated to the idempotent $1_n \in \nh$.
As an abelian group, $P_n$ is spanned by all diagrams with $n$ strands at the bottom.
The summand $\bigoplus\limits_{0 \le m \le n}\nhmn$ of $\cf$ forms a left $\nh$-submodule of $\cf$, naturally isomorphic to $P_n$.
So there is a natural isomorphism
$$\cf \cong \bigoplus\limits_{n \ge 1}P_{n}$$
of left $\nh$-modules.
In particular, $\cf$ is projective as a left $\nh$-module.

For $t \in \nhmn$, let $[t]_f$ denote the corresponding element of $\cf$, which satisfies
$$1_m \cdot [t]_f = [t]_f = [t]_f \cdot 1_{n-1}.$$
As a left $\nh$-module, $\cf$ is generated by $[1_n]_f$ for $n \ge 1$, since $[t]_f=t \cdot [1_n]_f$ for any $t \in \nhmn$.
The generator $[1_n]_f$ is represented by the diagram of $n$ vertical strands with a little bar at the lower end of the rightmost strand.
The right multiplication on the generators is given by
\begin{align} \label{eq right F}
[1_{n}]_f \cdot 1_{n-1}  =[1_{n}]_f, & \\
[1_{n}]_f \cdot a  =[(a \od 1_1)\cdot 1_{n}]_f, & \qquad \mbox{for}~~a\in NH_{n-1}, \\
[1_{n}]_f \cdot v_n  =[1_n \cdot (v_n \od 1_1)]_f=[v_{n,n+1}\cdot 1_{n+1}]_f, &
\end{align}
see figure \ref{fig alg10}.

\vspace{.2cm}
From now on, all tensor products are taken over $\nh$. We will simply write $\ot$ for $\ot_{\nh}$.
Let $\cf^k$ denote the $\nh$-bimodule given by the $k$-th tensor product $\cf^{\ot k}$ over $\nh$.
As an abelian group,
$$\cf^k= \bigoplus\limits_{m \ge 0, n \ge k}\nhmn,$$
where the $\nh$-bimodule structure is given by
$a\cdot f \cdot b=a~f~\rho^k(b)$, where $a~f~\rho^k(b)$ is the product in $\nh$, for $a,b \in \nh$ and $f \in \cf$.
Here, $\rho^k=\rho \circ \cdots \circ \rho: \nh \ra \nh$ maps $b$ to $b \od 1_k$ for $b \in \nh$.
The map $\rho^k$ is diagrammatically given by adding $k$ vertical strands on the right.
Let $[t]_{f^k} \in \cf^k$ which corresponds to $t \in \nhmn$ for $n \ge k$.
The collection $\{[1_n]_{f^k}~|~n\ge k\}$ generates $\cf^k$ as a left $\nh$-module.

The nilHecke algebra $NH_k$ naturally acts on $\cf^k$ on the right as follows.
For any $c \in NH_k$, define a map
\begin{align} \label{eq rt nh}
\begin{array}{cccc}
r(c): & \cf^k & \ra & \cf^k \\
& [t]_{f^k} & \mapsto & [t \cdot (1_{n-k} \od c)]_{f^k}
\end{array}
\end{align}
for $t \in \nhmn$.
Diagrammatically, the map $r(c)$ stacks $t \in NH_k$ onto the $k$ frozen strands from the bottom.
The action commutes with the left and right multiplication of $\nh$ on $\cf^k$.
Hence, $r(c)$ is a map of $\nh$-bimodules.

\vspace{.2cm}
Let $\mf{D}(\nh^e)$ denote the derived category of $\nh$-bimodules.
It is a monoidal triangulated category whose monoidal bifunctor is given by the derived tensor product over $\nh$.
The unit object, denoted by $\mb$, is isomorphic to $\nh$ as an $\nh$-bimodule placed in cohomological degree zero.
Since $\cf$ is projective as a left $\nh$-module, the derived tensor product reduces to the ordinary tensor product between $\cf$'s.
In particular, we view $\cf^k=\cf^{\ot k}$ placed in cohomological degree zero as an object of $\mf{D}(\nh^e)$.
The subscript $\mf{D}(\nh^e)$ will be omitted in $\Hom_{\mf{D}(\nh^e)}$ from now on.
We compute $\End(\mb)$ and $\End(\cf)$ in the following.

\begin{lemma} \label{lem end 1}
The endomorphism ring $\End(\mb)$ is isomorphic to $\K$ with a generator $id_{\mb}$.
\end{lemma}
\begin{proof}
For any $h \in \End(\mb)$, $h(1_n) \in NH_n$ lives in the center $Z(NH_n)$ which is isomorphic to the ring $\la_n$ of symmetric functions.
Moreover, $h(1_{n-1}) \cdot v_n=h(v_n)=v_n \cdot h(1_n) \in \nh^{n-1}_{n}$.
Suppose $h(1_0)=a \cdot 1_0 \in Z(NH_0)$ for some $a \in \K$.
Then $h(1_n)=a \cdot 1_n \in Z(NH_n)$ by induction on $n$.
So $h=a \cdot id_{\mb}$.
\end{proof}

There is a natural inclusion $\rho_n: NH_n \ra NH_{n+1}$ of algebras which sends $a$ to $a \od 1_1$.
Let $C_{n+1}(n)$ denote the centralizer of $NH_n$ in $NH_{n+1}$ with respect to the inclusion.
\begin{prop} \label{prop center}
There is an isomorphism $C_{n+1}(n) \cong Z(NH_n) \otimes \Z[x_{n+1}]$ of rings.
\end{prop}
\begin{proof}
The $NH_n$-bimodule $NH_{n+1}$ has a decomposition
$$_{NH_n}(NH_{n+1})_{NH_n} \cong (NH_n \otimes_{\K} \K[x_{n+1}]) \oplus (NH_n \otimes_{NH_{n-1}} NH_n),$$
where $NH_n \otimes_{\K} \K[x_{n+1}]$ is isomorphic to a sum of $\mathbb{N}$ copies of $NH_n$ as $NH_n$-bimodule, see figure \ref{fig alg15}.
The center of the bimodule $NH_n \otimes_{\K} \K[x_{n+1}]$ is isomorphic to $Z(NH_n) \otimes \Z[x_{n+1}]$.
It suffices to show that the center of the bimodule $NH_n \otimes_{NH_{n-1}} NH_n$ is trivial.

Note that $NH_n \otimes_{NH_{n-1}} NH_n$ is a free left $NH_n$-module with a basis
$$\{b_{i,k}=1_n \otimes (x_n^i  \pa_{[k]}) ~|~ i \ge 0, 0 \le k \le n-1\},$$
where $\pa_{[k]}=\pa_{n-1}\cdots \pa_{n-k}$ for $k \ge 1$ and $\pa_{[0]}=1_n$.
Let $c=\sum\limits_{i,k}c_{i,k}\cdot b_{i,k} \in NH_n \otimes_{NH_{n-1}} NH_n$, for $c_{i,k} \in NH_n$.
Let $k_0(c)=\max \{k ~|~ c_{i,k} \neq 0 ~\mbox{for some}~ i\}$, and $i_0(c)=\max \{i ~|~ c_{i,k_0(c)} \neq 0\}$.
Consider $c_1=x_{n-k_0} \cdot c$, and $c_2=c \cdot x_{n-k_0}$, for $x_{n-k_0} \in NH_n$.
Then $k_0(c_1)=k_0(c_2)=k_0(c)$, but $i_0(c_1)=i_0(c), i_0(c_2)=i_0(c)+1$.
Thus $c_1 \neq c_2$ which implies that the center is trivial.
\end{proof}

\begin{figure}[h]
\begin{overpic}
[scale=0.25]{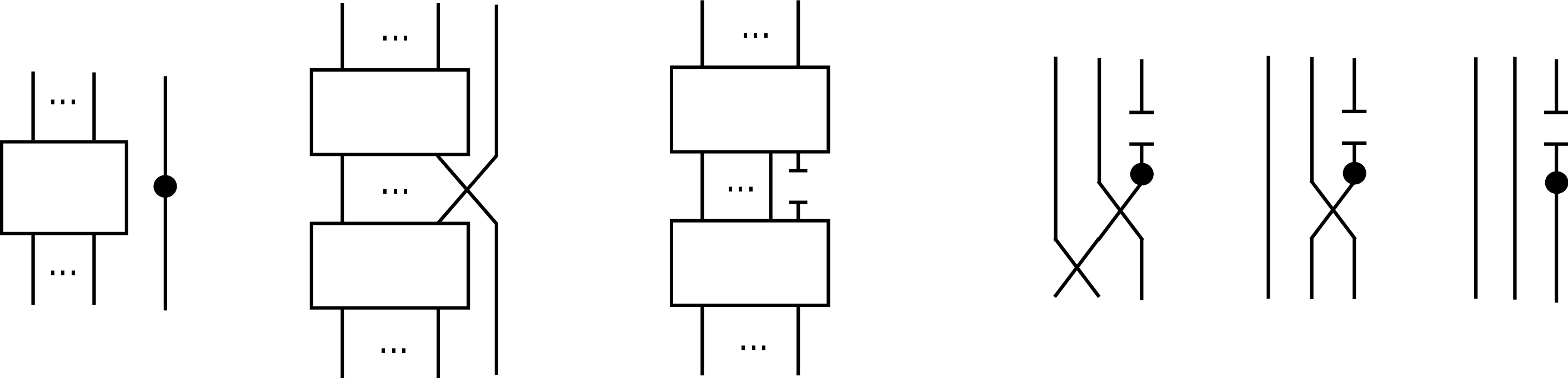}
\put(2,11){$\scriptstyle{NH_n}$}
\put(22,6){$\scriptstyle{NH_n}$}
\put(22,16){$\scriptstyle{NH_n}$}
\put(35,10){$\cong$}
\put(45,6){$\scriptstyle{NH_n}$}
\put(45,16){$\scriptstyle{NH_n}$}
\put(69,0){$b_{i,2}$}
\put(82,0){$b_{i,1}$}
\put(94,0){$b_{i,0}$}
\put(74,12){$i$}
\put(88,12){$i$}
\put(101,12){$i$}
\end{overpic}
\caption{The $NH_n$-bimodule $NH_{n+1}$.}
\label{fig alg15}
\end{figure}

Proposition \ref{prop center} together with a similar argument as in the proof of Lemma \ref{lem end 1} gives the following result.
\begin{cor}
There is a ring isomorphism $\K[x] \ra \End(\cf)$ which maps $c \in \K[x]$ to $r(c) \in \End(\cf)$.
\end{cor}

\subsection{The complex lifting $\exp(-f)$}
The goal is to lift the expansion
$$\exp(-f)=\sum\limits_{k\ge 0}(-1)^k\frac{f^k}{k!},$$
to a complex in $\mf{D}(\nh^e)$, where $1$ and $f$ are lifted to the objects $\mb$ and $\cf$, respectively.
We use certain direct summands of $\cf^k$ to lift the divided powers $\frac{f^k}{k!}$.
To define the differential, we will use the new generators with short strands.

Recall some basic facts about some idempotents of $NH_k$ as follows.
Let $w_0(k) \in S_k$ denote the longest element of $S_k$, and $\pak=\pa_{w_0(k)} \in NH_k$.
There are idempotents
$$e_k=x_1^{k-1}\cdots x_{k-1} \pak$$
in $NH_k$.
Let $e_{i,k} = 1_{i-1} \od e_2 \od 1_{k-i-1} \in NH_k$ denote the diagram obtained from $e_2$ by adding $i-1$ and $k-i-1$ vertical strands on the left and right, respectively, see figure \ref{fig alg9}.
It is easy to see that $e_k = e_{i_1, k} \cdots e_{i_l, k}$, where $s_{i_1}\cdots s_{i_l}$ is a reduced expression of $w_0(k) \in S_k$.
In other words, $e_2$ is the building block of $e_k$.

\begin{figure}[h]
\begin{overpic}
[scale=0.3]{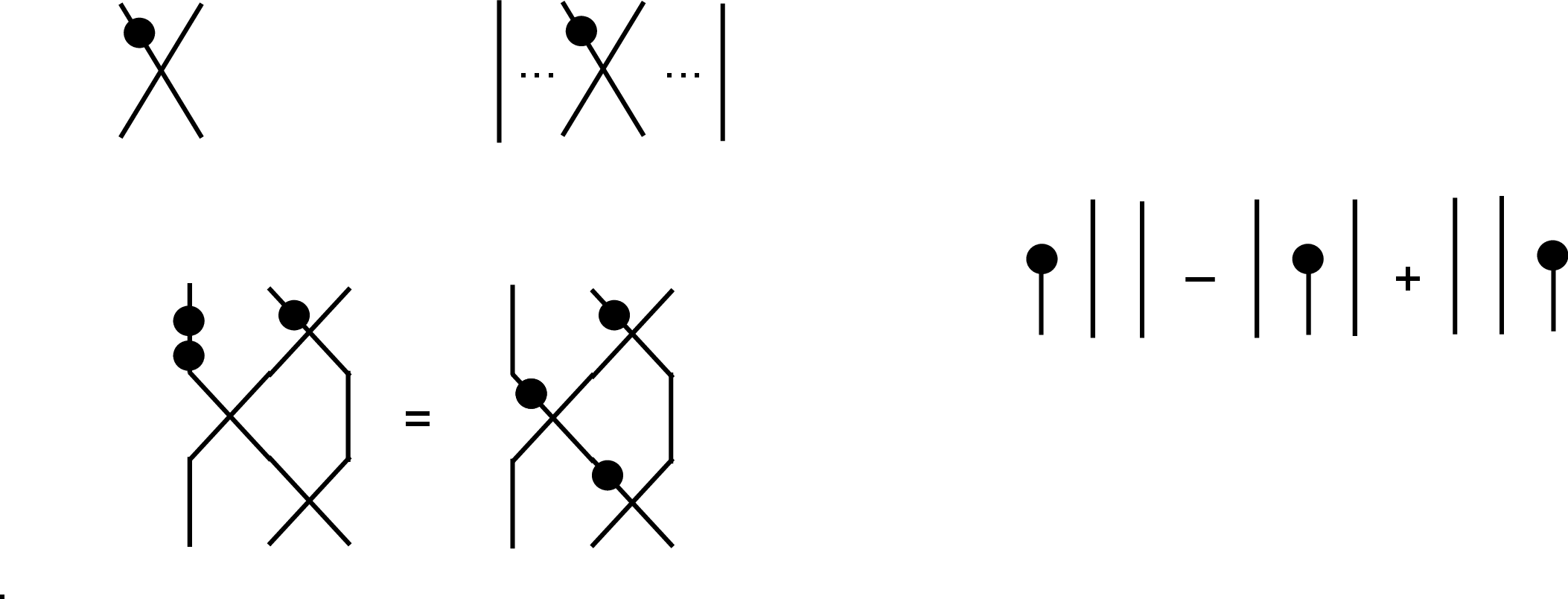}
\put(0,32){$e_2:$}
\put(22,32){$e_{i,k}:$}
\put(58,18){$u_3:$}
\put(35,26){$\scriptstyle{i}$}
\put(39,26){$\scriptstyle{i+1}$}
\end{overpic}
\caption{The diagrams for $e_2, e_{i,k}, e_3=e_{2,3}e_{1,3}e_{2,3}$, and $u_3$.}
\label{fig alg9}
\end{figure}

The idempotent $e_k$ induces an idempotent endomorphism $r(e_k)$ of the bimodule $\cf^k$.

\begin{defn} \label{def fk}
Define the $\nh$-bimodule $\cfk$ as a direct summand of $\cf^k$ corresponding to the idempotent endomorphism $r(e_k)$.
Elements of $\cfk$ are of the form $[t \cdot (1_{n-k} \od e_k)]_{f^k}$ for $t \in \nhmn$.
\end{defn}

Define an alternating sum
\begin{gather} \label{def u}
u_k = \sum\limits_{i=1}^{k}(-1)^{i-1}v_{i,k} \in \nh^{k-1}_k,
\end{gather}
for $k \ge 1$, see figure \ref{fig alg9} for $u_3$.
Then $u_k \cdot u_{k+1}=0$ by the isotopy relation of disjoint short strands.

Via the right action of $NH_k$ on $\cf^k$ in (\ref{eq rt nh}), the element $u_k$ induces a map
\begin{align} \label{eq rt u1}
\begin{array}{cccc}
r(u_k): & \cf^{k-1} & \ra & \cf^k \\
& [t]_{f^{k-1}} & \mapsto & [t \cdot (1_{n+1-k} \od u_{k})]_{f^k}
\end{array}
\end{align}
for $t \in \nhmn$, $n \ge k-1$.

Consider the restriction of $r(u_k)$ to $\cf^{(k-1)}$ followed by the projection of $\cf^k$ onto $\cfk$.
The resulting map is denoted by $d(u_k)$:
\begin{align} \label{eq rt u}
\begin{array}{cccc}
d(u_k): & \cf^{(k-1)} & \ra & \cfk \\
& [t \cdot (1_{n+1-k} \od e_{k-1})]_{f^{k-1}} & \mapsto & [t \cdot (1_{n+1-k} \od (e_{k-1}u_{k}e_{k})]_{f^k}
\end{array}
\end{align}
for $t \in \nhmn$, $n \ge k-1$.

We compute $d(u_{k+1}) \circ d(u_k)$ in the following.

\begin{lemma} \label{lem eue}
There is an equality $e_{k-1} u_k e_k=e_{k-1} u_k \in \nh^{k-1}_k$.
\end{lemma}
\begin{proof}
We claim that $e_{k-1} u_k e_{i,k}=e_{k-1} u_k$ for all $1 \le i \le k-1$.
The lemma follows from the claim since $e_k$ can be written as a product of $e_{i, k}$'s.

To prove the claim, we compute
\begin{align*}
e_{k-1} u_k \pa_{i,k} & = \sum\limits_{j=1}^{k}(-1)^{j-1}e_{k-1} v_{j,k} \pa_{i,k} \\
& = \sum\limits_{j \neq i,i+1}(-1)^{j-1} e_{k-1} v_{j,k} \pa_{i,k} + (-1)^{i-1} e_{k-1} v_{i,k} \pa_{i,k} + (-1)^{i} e_{k-1} v_{i+1,k} \pa_{i,k}.
\end{align*}
Each term of the first summation is zero since $v_{j,k} \pa_{i,k}=\pa_{i',k}v_{j,k}$ for some $i'$ when $j \neq i,i+1$, and $e_{k-1} \pa_{i',k-1}=0$ for all $i'$.
The remaining two terms cancel each other since $v_{i,k} \pa_{i,k}=v_{i+1,k} \pa_{i,k}$ from the exchange relation in Definition \ref{def nh} (3).
So $e_{k-1} u_k \pa_{i,k}=0$.
By the nilHecke relation, $e_2=1_2 + \pa_1 x_2$ so that $e_{i,k}=1_k + \pa_{i,k} x_{i+1,k}$.
Thus $$e_{k-1} u_k e_{i,k}=e_{k-1} u_k (1_k + \pa_{i,k} x_{i+1,k})= e_{k-1} u_k 1_k= e_{k-1} u_k.$$
The lemma follows.
\end{proof}

\begin{rmk}
In general, $e_{k-1} u_k e_k \neq u_k e_k$.
\end{rmk}

The lemma above implies that
\begin{gather} \label{eq u2}
(e_{k-1} u_k e_k) \cdot (e_{k} u_{k+1} e_{k+1})=e_{k-1} u_k e_k u_{k+1} e_{k+1} = e_{k-1} u_k u_{k+1} e_{k+1}=0,
\end{gather}
where the last equality holds because $u_k u_{k+1}=0$.
Hence, $d(u_{k+1}) \circ d(u_k): \cf^{(k-1)} \ra \cf^{(k+1)}$ is zero by the definition of $d(u_k)$ in (\ref{eq rt u}).
We define a complex
\begin{gather} \label{def e^-f}
\exp(-\cf)=\left(\bigoplus\limits_{k \ge 0} \cfk[-k], d=\bigoplus\limits_{k \ge 1}d_k\right),
\end{gather}
where the components of the differential $d$ are given by $d_k=d(u_k): \cf^{(k-1)} \ra \cf^{(k)}$.

\subsection{The complex lifting $\exp(f)$}
The goal is to lift the expansion $$\exp(f)=\sum\limits_{k\ge 0}\frac{f^k}{k!},$$
to an object in $\mf{D}(\nh^e)$.
We use another direct summand of $\cf^k$ induced by an idempotent which is different from $e_k$.
The differential is induced by certain elements in the extension group $\Ext^1(\cf^k, \cf^{k-1})$.

\vspace{.2cm}
In the following, we construct an $\nh$-bimodule $\cg$ in two steps.
We will show that $\cg$ is an extension of $\cf$ by $\mb=\nh$ as $\nh$-bimodules:
$$0 \ra \mb \ra \cg \ra \cf \ra 0.$$

\vspace{.1cm}
\n{\bf Step 1: The left module.}
We first define
$$\cg=\mb \oplus \cf$$
as a left $\nh$-module.
Let $[t]_1 \in \mb$, and $[t']_f \in \cf$ for $t,t' \in \nhmn$.
Then
$$t \cdot [1_n]_1=[t]_1, \qquad t' \cdot [1_n]_f=[t']_f.$$
The left $\nh$-module $\cg$ is projective, and generated by $[1_n]_1$ and $[1_{n+1}]_f$, for $n \ge 0$.

\vspace{.1cm}
\n{\bf Step 2: The right module.} The right multiplication is defined on the generators $[1_n]_1$ and $[1_{n+1}]_f$ as follows.
The summand $\mb$ is a right $\nh$-submodule of $\cg$:
$$[1_n]_1 \cdot t=[t]_1, \quad \mbox{for}~~ t \in \nh^n_{n'}.$$
The algebra $\nh$ contains $\bigoplus\limits_{n \ge 0}NH_n$ as a subalgebra.
The summand $\cf$ is a right $\bigoplus\limits_{n \ge 0}NH_n$-submodule of $\cg$:
$$[1_{n+1}]_f \cdot t=[t \od 1_1]_f=(t \od 1_1) \cdot [1_{n+1}]_f, \quad \mbox{for}~~ t \in NH_n.$$
The only nontrivial part of the definition of $\cg$ is the right multiplication on $[1_{n}]_f$ with $v_n \in \nh^{n-1}_n$, for $n \ge 1$:
\begin{gather} \label{eq rt g}
[1_{n}]_f \cdot v_n=[v_n \od 1_1]_f + [1_n]_1=[v_{n,n+1}]_f+[1_n]_1,
\end{gather}
see figure \ref{fig alg11}.
Here we use the presentation of $\nh$ in Definition \ref{def nh'}.
The algebra $\nh$ is generated by $NH_{n-1}$ and $v_n$ for $n \ge 1$.
In particular, $v_{i,n}=v_n \cdot g_{i,n}$ for some element $g_{i,n} \in NH_n$.
Define $[1_{n}]_f \cdot v_{i,n}=([1_{n}]_f \cdot v_n) \cdot g_{i,n}$.

Define the right multiplication on $[t]_f$ for $t \in \nhmn$, $b \in \nh$ as
$$[t]_f \cdot b=t \cdot ([1_n]_f \cdot b).$$
Our construction of $\cg$ is complete.

As a left $\nh$-module, $\cg$ is projective, and generated by $[1_n]_1$ and $[1_{n+1}]_f$ for $n \ge 0$.
Thus, the right multiplication on $\cg$ is determined by the right multiplication on the generators $[1_n]_1$ and $ [1_{n+1}]_f$.
By definition, it commutes with the left multiplication.

\begin{figure}[h]
\begin{overpic}
[scale=0.3]{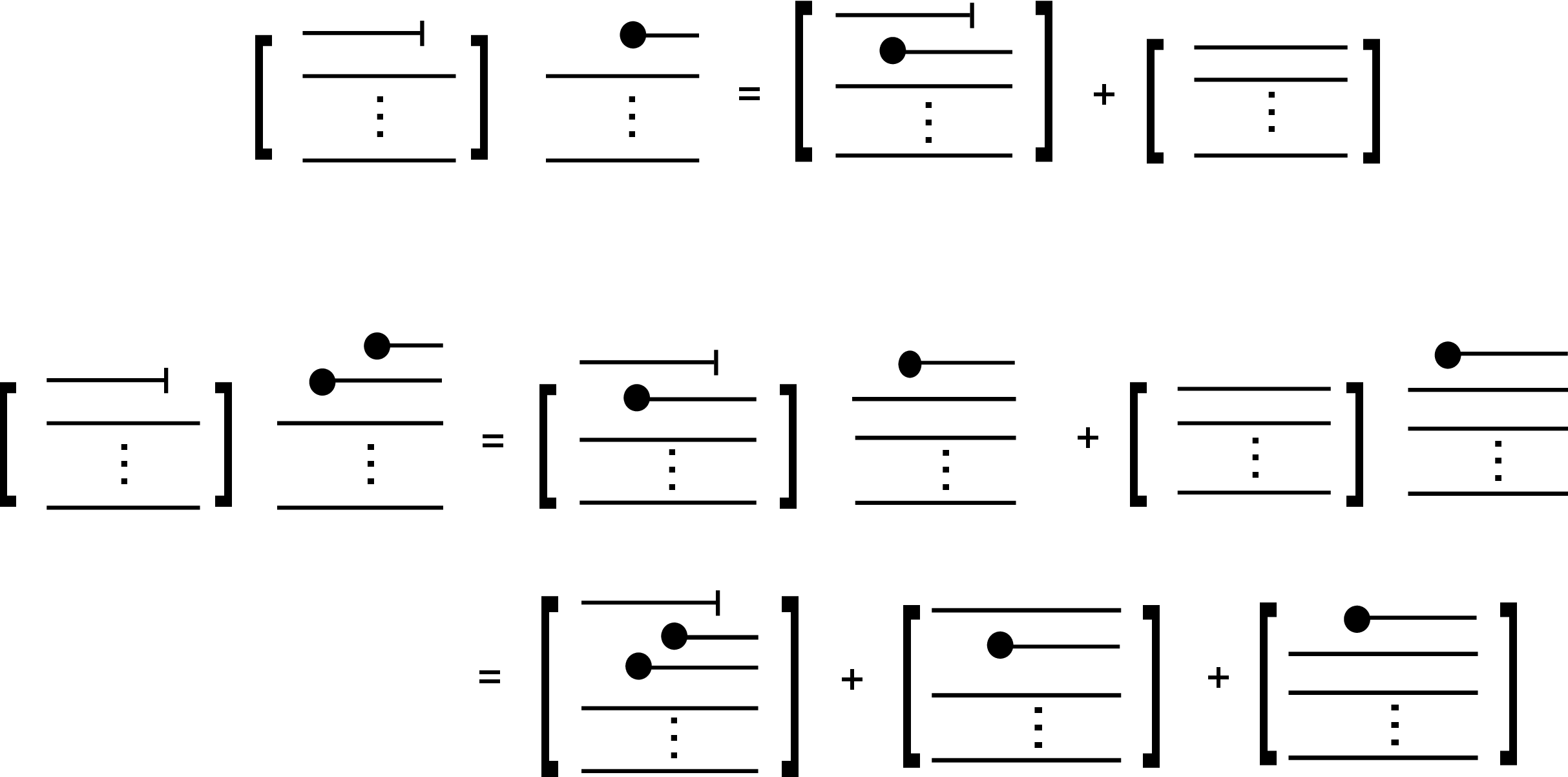}
\put(31,39){$f$}
\put(67,39){$f$}
\put(88,39){$1$}
\put(15,17){$f$}
\put(51,17){$f$}
\put(87,17){$1$}
\put(51,0){$f$}
\put(74,0){$1$}
\put(97,0){$1$}
\end{overpic}
\caption{Diagrams are rotated by $\frac{\pi}{2}$ counterclockwise from the vertical direction to the horizontal direction. The upper part is the definition $[1_{n}]_f \cdot v_n$; the lower part computes $([1_{n}]_f \cdot v_n) \cdot v_{n+1}$.}
\label{fig alg11}
\end{figure}

\begin{lemma} \label{lem rt g}
The right multiplication on $\cg$ is well-defined.
\end{lemma}
\begin{proof}
Since the action of $NH_n$ on $\cg$ is the ordinary multiplication on $\nh$, we only have to check the relations involving $v_n$ in Definition \ref{def nh'} (3).

We check the isotopy relation of a short strand with a dot
\begin{align*}
([1_{n}]_f \cdot v_n) \cdot x_{j,n} &=([v_{n,n+1}]_f+[1_n]_1) \cdot x_{j,n} \\
&=[v_{n,n+1} \cdot (x_{j,n}\od 1_1)]_f+[x_{j,n}]_1 \\
&=[x_{j,n} \cdot v_{n,n+1}]_f+[x_{j,n}]_1 \\
&=([1_{n}]_f \cdot x_{j,n-1}) \cdot v_n.
\end{align*}
The proof of the isotopy relation of a short strand with a crossing is similar, and left to the reader.

We check the relation (F) in Definition \ref{def nh'} (3): $v_{n} v_{n+1}=v_{n} v_{n+1} s_{n,n+1}$
\begin{align*}
([1_{n}]_f \cdot v_n) \cdot v_{n+1} &=([v_{n,n+1}]_f+[1_n]_1) \cdot v_{n+1} \\
&=[v_{n,n+1} \cdot 1_{n+1}] \cdot v_{n+1}+[v_{n+1}]_1 \\
&=[v_{n,n+1} \cdot v_{n+1,n+2}]_f+[v_{n,n+1} \cdot 1_{n+1}]_1+[v_{n+1}]_1 \\
&=[v_{n,n+1} \cdot v_{n+1,n+2}]_f+[v_{n,n+1}+v_{n+1}]_1,
\end{align*}
see figure \ref{fig alg11}.
Recall that $s_{n,n+1} \in NH_{n+1}$ exchanges $v_{n+1}$ and $v_{n,n+1}$.
So
\begin{align*}
([1_{n}]_f \cdot v_n \cdot v_{n+1}) \cdot s_{n,n+1}&=[v_{n,n+1} \cdot v_{n+1,n+2}]_f \cdot s_{n,n+1} +[v_{n,n+1}+v_{n+1}]_1 \cdot s_{n,n+1} \\
&= [v_{n,n+1} \cdot v_{n+1,n+2} \cdot (s_{n,n+1} \od 1_1)]_f +[v_{n,n+1}+v_{n+1}]_1 \\
&= [v_{n,n+1} \cdot v_{n+1,n+2}]_f+[v_{n,n+1}+v_{n+1}]_1.
\end{align*}
This proves that the right multiplication on $\cg$ is well-defined.
\end{proof}


The $\nh$-bimodule $\cg$ fits into a short exact sequence of $\nh$-bimodules:
$$0 \ra \mb \ra \cg \ra \cf \ra 0.$$

\begin{lemma} \label{lem g nontriv}
The extension $\cg$ is not split.
\end{lemma}
\begin{proof}
Suppose $\cg$ is split, i.e. there is a commutative diagram
$$
\xymatrix{
 0 \ar[r] & \mb \ar[r] \ar[d]^{id} &  \cg \ar[r] \ar[d]^{\psi} & \cf \ar[d]^{id} \ar[r] & 0. \\
 0 \ar[r] & \mb \ar[r]  & \mb\oplus\cf  \ar[r] & \cf \ar[r] & 0. \\
}$$
of $\nh$-bimodules.
Let $[1_{n-1}]'_1$ and $[1_n]'_f$ for $n \ge 1$ denote the generators of $\mb \oplus \cf$.
The commutative diagram implies that $\psi([1_{n-1}]_1)=[1_{n-1}]'_1$, and $\psi([1_n]_f)=[1_n]'_f+[t]'_1$ for some $t \in \nh^{n}_{n-1}$.
Since $\nh^{n}_{n-1}=0$, we have $\psi([1_n]_f)=[1_n]'_f$.
As a left $\nh$-module, $\cg$ is generated by $[1_{n-1}]_1$ and $[1_n]_f$.
So $\psi([t]_f)=\psi(t \cdot [1_n]_f)=t \cdot \psi([1_n]_f)=t \cdot [1_n]'_f=[t]'_f$.
Similarly, $\psi([t]_1)=[t]'_1$.
On the other hand,
$$\psi([1_n]_f \cdot v_n)=\psi([v_{n,n+1}]_f+[1_n]_1)=[v_{n,n+1}]'_f+[1_n]'_1, \qquad \psi([1_n]_f) \cdot v_n=[1_n]'_f \cdot v_n=[v_{n,n+1}]'_f.$$
Thus $\psi$ is not a map of right $\nh$-modules. This is a contradiction.
\end{proof}

\vspace{.2cm}
The extension $\cg$ gives rise to a morphism $\wtv_1 \in \Hom(\cf,\mb[1])$ in the derived category $\mf{D}(\nh^e)$.
We write $\Hom^1(\cf,\mb)$ for $\Hom(\cf,\mb[1])$.
Lemma \ref{lem g nontriv} implies that $\wtv_1 \neq 0 \in \Hom^1(\cf,\mb)$.

Define $$\wtv_{i,n}=1_{\cf^{i-1}} \ot \wtv_1 \ot 1_{\cf^{n-i}} \in \Hom^1(\cf^n,\cf^{n-1}),$$
for $1 \le i \le n$.

Two elements $\wtv_{1,2}$ and $\wtv_{2,2}$ correspond to two extensions $\cgf$ and $\cfg$ of $\cf^2$ by $\cf$:
$$\wtv_{1,2}: \qquad 0 \ra \cf \ra \cgf \ra \cf^2 \ra 0,$$
$$\wtv_{2,2}: \qquad 0 \ra \cf \ra \cfg \ra \cf^2 \ra 0.$$

Recall $s_{1,2} \in NH_2$ in (\ref{def s}) and the map $r(s_{1,2}): \cf^2 \ra \cf^2$ of $\nh$-bimodules in (\ref{eq rt nh}).

\begin{lemma} \label{lem ext rel}
There exists a map $\phi: \cgf \ra \cfg$ of $\nh$-bimodules such that the following diagram commutes
$$
\xymatrix{
\wtv_{1,2}: & 0 \ar[r] & \cf \ar[r] \ar[d]^{id} & \cgf \ar[r] \ar[d]^{\phi} & \cf^2 \ar[d]^{r(s_{1,2})} \ar[r] & 0. \\
\wtv_{2,2}: & 0 \ar[r] & \cf \ar[r]  & \cfg \ar[r] & \cf^2 \ar[r] & 0. \\
}$$
\end{lemma}
\begin{proof}
The left projective $\nh$-module $\cgf \cong \cf \oplus \cf^2$ is generated by elements
$$[1_{n}]_{f}^{gf}:=[1_{n}]_1^g \ot [1_{n}]_f, \qquad [1_{n+1}]_{f^2}^{gf}:=[1_{n+1}]_f^g \ot [1_{n}]_f,$$
for $n \ge 1$. Similarly, the left projective $\nh$-module $\cfg \cong \cf \oplus \cf^2$ is generated by elements
$$[1_{n}]_f^{fg}:=[1_{n}]_f \ot [1_{n-1}]_1^g, \qquad [1_{n+1}]_{f^2}^{fg}:=[1_{n+1}]_f \ot [1_{n}]_f^g,$$
for $n \ge 1$.
Define a map $\phi: \cgf \ra \cfg$ of left $\nh$-modules on the generators as
$$\phi([1_{n}]_{f}^{gf})=[1_{n}]_f^{fg}, \qquad \phi([1_{n+1}]_{f^2}^{gf})=[s_{n,n+1}]_{f^2}^{fg},$$
see figure \ref{fig alg14}.
The map $\phi$ makes the diagram commute.
\begin{figure}[h]
\begin{overpic}
[scale=0.25]{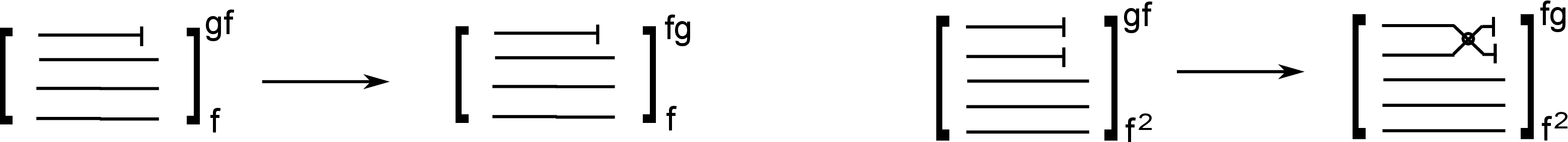}
\put(78,6){$\phi$}
\put(20,6){$\phi$}
\end{overpic}
\caption{The map $\phi$ defined on the generators $[1_{n}]_{f}^{gf}$ and $[1_{n+1}]_{f^2}^{gf}$.}
\label{fig alg14}
\end{figure}

We need to show that $\phi$ is a map of right $\nh$-modules.
It is true when restricting to the right $\nh$-submodule $\cf$ of $\cgf$.
It remains to show that $\phi([1_{n+1}]_{f^2}^{gf} \cdot v_n)=\phi([1_{n+1}]_{f^2}^{gf}) \cdot v_n$.
We compute
\begin{align*}
[1_{n+1}]_{f^2}^{gf} \cdot v_n & = [1_{n+1}]_f^g \ot ([1_{n}]_f \cdot v_n) \\
& = [1_{n+1}]_f^g \ot [v_n \od 1_1]_f \\
& = ([1_{n+1}]_f^g \cdot v_{n,n+1} ) \ot [1_{n+1}]_f \\
& = ([1_{n+1}]_f^g \cdot v_{n+1} \cdot s_{n,n+1} ) \ot [1_{n+1}]_f \\
& = (([v_{n+1} \od 1_1]_f^g + [1_{n+1}]_1^g) \cdot s_{n,n+1} ) \ot [1_{n+1}]_f \\
& = ([v_{n+1,n+2} \cdot s_{n,n+2}]_f^g + [s_{n,n+1}]_1^g ) \ot [1_{n+1}]_f \\
& = ([v_{n,n+2}]_f^g + [s_{n,n+1}]_1^g ) \ot [1_{n+1}]_f \\
& = [v_{n,n+2}]_{f^2}^{gf} + [s_{n,n+1}]_f^{gf}.
\end{align*}
So $\phi([1_{n+1}]_{f^2}^{gf} \cdot v_n)=[v_{n,n+2} \cdot s_{n,n+1}]_{f^2}^{fg} + [s_{n,n+1}]_f^{fg}$.
On the other hand,
\begin{align*}
\phi([1_{n+1}]_{f^2}^{gf}) \cdot v_n & = [s_{n,n+1}]_{f^2}^{fg} \cdot v_n \\
& = [s_{n,n+1}]_f \ot ([1_{n}]_f^g \cdot v_n) \\
& = [s_{n,n+1}]_f \ot ([v_{n,n+1}]_f^g+[1_n]_1^g) \\
& = ([s_{n,n+1}]_f \cdot v_{n,n+1}) \ot [1_{n+1}]_f^g + [s_{n,n+1}]_f^{fg} \\
& = [s_{n,n+1} \cdot v_{n,n+2}]_{f^2}^{fg} + [s_{n,n+1}]_f^{fg}.
\end{align*}
Since $v_{n,n+2} \cdot s_{n,n+1}=s_{n,n+1} \cdot v_{n,n+2}$, it follows that
$\phi([1_{n+1}]_{f^2}^{gf} \cdot v_n)=\phi([1_{n+1}]_{f^2}^{gf}) \cdot v_n$.
A graphic counterpart of this computation is depicted in figure \ref{fig alg12}.
\end{proof}

\begin{figure}[h]
\begin{overpic}
[scale=0.25]{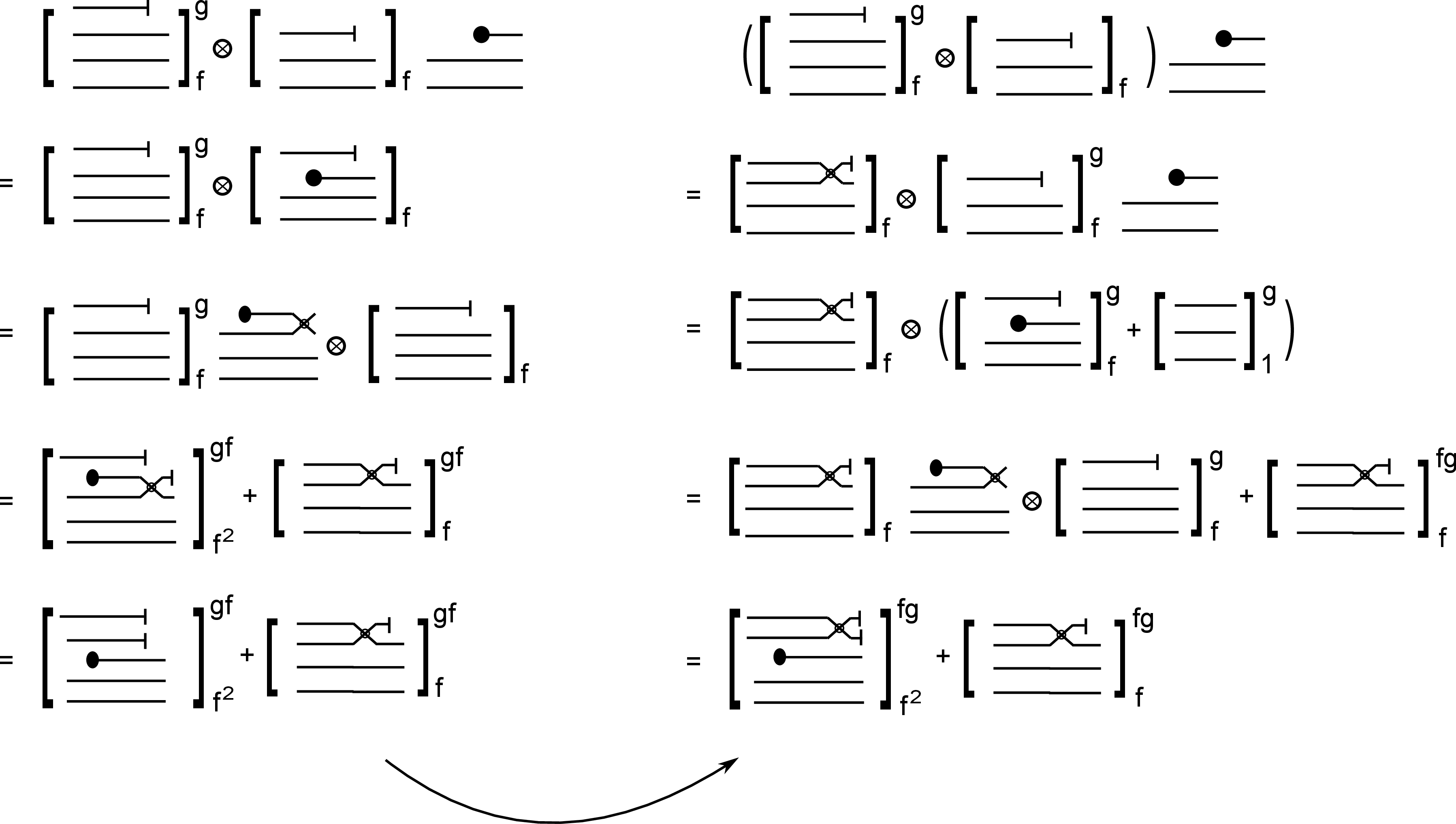}
\put(48,52){$\phi$}
\put(37,3){$\phi$}
\end{overpic}
\caption{The pictures compute $\phi([1_{n+1}]_{f^2}^{gf} \cdot v_n)$ and $\phi([1_{n+1}]_{f^2}^{gf}) \cdot v_n$ for $n=3$ on the left and right, respectively.}
\label{fig alg12}
\end{figure}

Recall that $r(s_{1,2}) \in \Hom^0(\cf^2,\cf^2)$, and $\wtv_{1,2} \in \Hom^1(\cf^2,\cf)$.
So $\wtv_{1,2} \circ r(s_{1,2}) \in \Hom^1(\cf^2,\cf)$ corresponds to another extension of $\cf^2$ by $\cf$, denoted by $s(\cgf)$.
The lemma above shows that the extensions $s(\cgf)$ and $\cfg$ are equivalent.

\begin{cor} \label{cor ext rel}
The equalities $\wtv_{2,2}=\wtv_{1,2}\circ r(s_{1,2})$, and $\wtv_{2,2} \circ r(\pa_{1,2})=-\wtv_{1,2}\circ r(\pa_{1,2})$ hold in $\Hom^1(\cf^2,\cf[1])$.
\end{cor}
\begin{proof}
The first equality directly follows from Lemma \ref{lem ext rel}.
Precomposing the first equality with $r(\pa_{1,2})$ gives the second one, since $r(s_{1,2}) \circ r(\pa_{1,2})=r(\pa_{1,2} \cdot s_{1,2})=r(-\pa_{1,2})$, see figure \ref{fig alg5}.
\end{proof}

In the monoidal category $\mf{D}(\nh^e)$, we could diagrammatically represent $\wtv_1 \in \Hom^1(\cf,\mb)$ as a short strand with one endpoint at the top and one endpoint in the middle decorated by a circle, see figure \ref{fig alg13}.
It is of cohomological degree one.
The element $\wtv_{i,n}$ can be obtained from $\wtv_1$ by adding $i-1$ and $n-i$ vertical strands on the left and right, respectively.
The two relations in Corollary \ref{cor ext rel} are depicted in figure \ref{fig alg13}.
They are the analogues of the slide relation, and the exchange relation in $\Hom^0(\cf,\cf^{2})$, see figures \ref{fig alg4} and \ref{fig alg2}.

\begin{figure}[h]
\begin{overpic}
[scale=0.25]{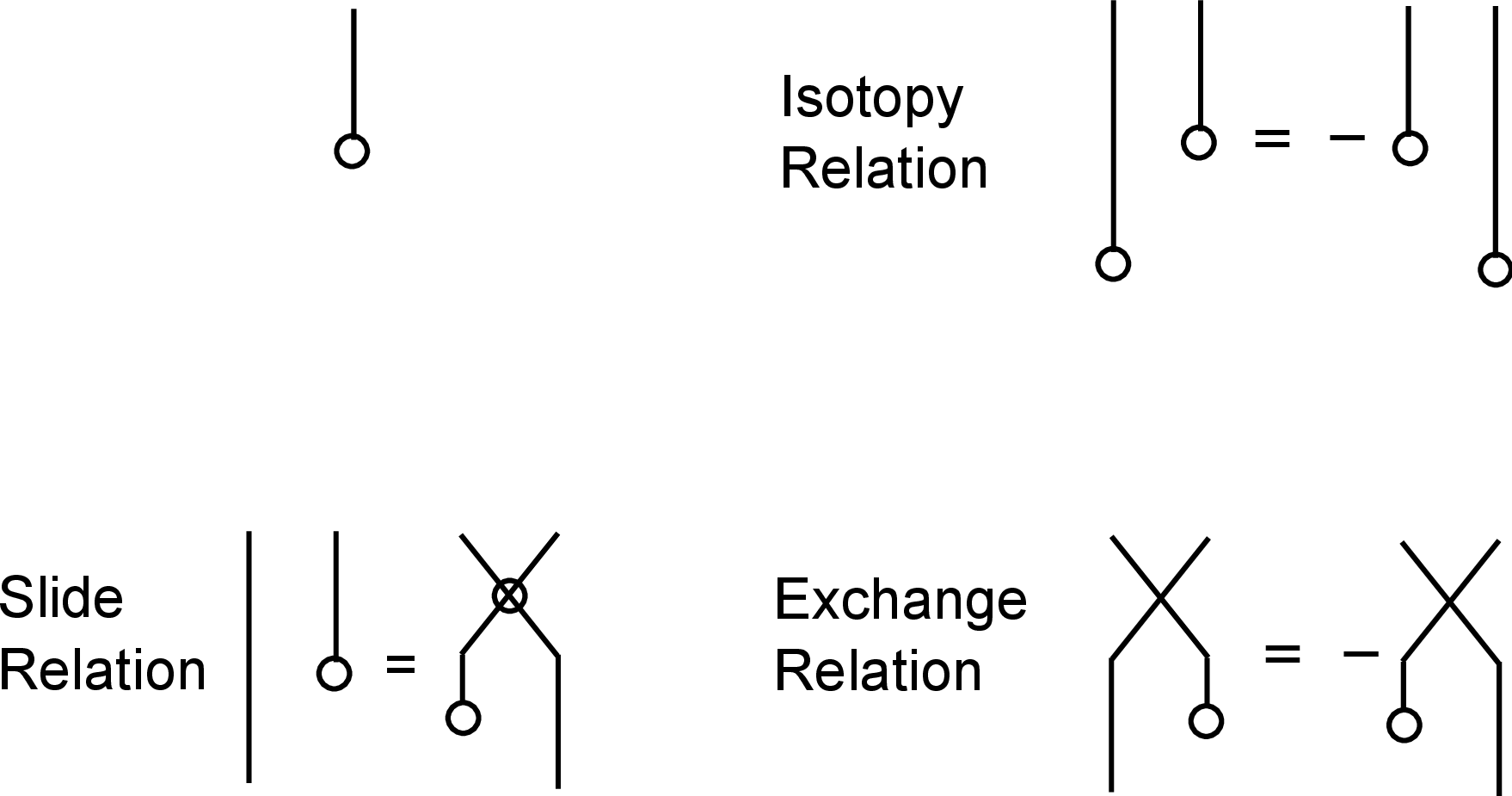}
\put(10,45){$\wtv_1:$}
\end{overpic}
\caption{A graphic presentation of $\wtv_1 \in \Hom^1(\cf,\mb[1])$, and the slide, exchange, and isotopy relations in $\Hom^1(\cf^2,\cf)$.}
\label{fig alg13}
\end{figure}

The super version of the isotopy relation of disjoint diagrams holds in $\mf{D}(\nh^e)$:
$$(a \ot 1_{N'}) \circ (1_{M} \ot b)=(-1)^{|a||b|}(1_{N} \ot b) \circ (a\ot 1_{M'}) \in \Hom^{|a|+|b|}(M\ot M',N \ot N'),$$
for $a \in  \Hom^{|a|}(M,N), b \in  \Hom^{|b|}(M',N')$.
In particular, $\wtv_{2,2} \circ \wtv_1=-\wtv_{1,2} \circ \wtv_1$ from the isotopy of two disjoint short strands of degree one.

As an analogue of $u_k$ defined in (\ref{def u}), define a sum
\begin{gather} \label{def wtu}
\wtu_k = \sum\limits_{i=1}^{k}\wtv_{i,k} \in \Hom^1(\cf^k,\cf^{k-1}),
\end{gather}
for $k \ge 1$.
Then $\wtu_k \circ \wtu_{k+1}=0$ by the super isotopy relation of disjoint short strands.

Define the idempotent
$$\wte_k=\pak x_1^{k-1}\cdots x_{k-1} \in NH_k.$$
It is the flip of $e_k$ with respect to the horizontal axis.
There is an induced idempotent endomorphism $r(\wte_k) \in \Hom^0(\cf^k, \cf^k)$.
Define $\wtfk$ to be the direct summand of $\cf^k$ corresponding to the idempotent endomorphism $r(\wte_k)$.

The morphism $\wtu_k \in \Hom^1(\cf^k,\cf^{k-1})$ induces a morphism in $\Hom^1(\wt{\cf}^{(k)},\wt{\cf}^{(k-1)})$ as the restriction of $\wtu_k$ to $\wt{\cf}^{(k)}$ followed by a projection onto $\wt{\cf}^{(k-1)}$.
Let $d(\wtu_k)$ denote the resulting morphism.

\begin{lemma} \label{lem eue'}
There is an equality $r(\wte_{k-1}) \circ \wtu_k \circ r(\wte_k)=r(\wte_{k-1}) \circ \wtu_k \in \Hom^1(\cf^k,\cf^{k-1})$.
\end{lemma}
\begin{proof}
The proof is similar to that of Lemma \ref{lem eue} except that there is a minus sign in the exchange relation $\wtv_{2,2} \circ r(\pa_{1,2})=-\wtv_{1,2}\circ r(\pa_{1,2})$, as in Corollary \ref{cor ext rel}.
\end{proof}

The lemma above implies that
\begin{align*}
&(r(\wte_{k-1}) \circ \wtu_k \circ r(\wte_k)) \circ (r(\wte_{k}) \circ \wtu_{k+1} \circ r(\wte_{k+1})) \\
=&r(\wte_{k-1}) \circ \wtu_k \circ r(\wte_{k}) \circ \wtu_{k+1} \circ r(\wte_{k+1}) \\
=&r(\wte_{k-1}) \circ \wtu_k  \circ \wtu_{k+1} \circ r(\wte_{k+1})=0,
\end{align*}
where the last equality holds because $\wtu_k \circ \wtu_{k+1}=0$.

Hence, $d(\wtu_{k}) \circ d(\wtu_{k+1}) \in \Hom^2(\wt{\cf}^{(k+1)}, \wt{\cf}^{(k-1)})$ is zero.
We define a complex
\begin{gather} \label{def e^f}
\exp(\cf)=\left(\bigoplus\limits_{k \ge 0} \wtfk, d=\bigoplus\limits_{k \ge 1}d_k\right),
\end{gather}
where the components of the differential $d$ are given by $d_k=d(\wtu_k) \in \Hom^1(\wt{\cf}^{(k)}, \wt{\cf}^{(k-1)})$.

\section{Discussions} \label{Sec dis}
\n {\bf  Non-invertibility of $\exp(\cf)$ and $\exp(-\cf)$.}

Suppose that $\exp(\cf), \exp(-\cf)$ are invertible objects.
They should induce invertible endofunctors via categorical actions.
We will consider certain actions of $\mf{D}(\nh^e)$ or its variants, and show that the induced functors cannot be invertible.

The derived tensor product over $\nh$ induces an action of $\mf{D}(\nh^e)$ on the derived category $\mf{D}(\nh)$ of left $\nh$-modules.
Let $P_{(k)}=\nh \cdot e_k$ denote the left projective $\nh$-module.
There are isomorphisms
$$\cf^{(k)}(P_0) \cong \wt{\cf}^{(k)}(P_0) \cong P_{(k)} \in \mf{D}(\nh).$$

By definition, $\exp(\cf)$ is an iterated extension of $\wt{\cf}^{(k)}$'s.
So $\exp({\cf})(P_0)$ is isomorphic to an iterated extension of $\wt{\cf}^{(k)}(P_0)$'s.
But there is no nontrivial extension between $\wt{\cf}^{(k)}(P_0)$'s since $\wt{\cf}^{(k)}(P_0) \cong P_{(k)}$ are projective for all $k$.
Thus,
$$\exp({\cf})(P_0) \cong \bigoplus\limits_{k=0}^{\infty}P_{(k)}.$$
The induced map: $\End(P_0) \ra \End(\exp({\cf})(P_0))$ is very far from being surjective.
Hence, $\exp({\cf})$ cannot be an invertible endofunctor of $\mf{D}(\nh)$, nor an invertible object in $\mf{D}(\nh^e)$.

In the case of $\exp(-\cf)$, we consider a cyclotomic quotient $\nh(1)$ of $\nh$, where the two-sided ideal under quotient is generated by $x_{1,n}$ for $n>0$.
Idempotents $1_n$ are in the ideal for $n>1$, and $\nh(1)$ has a $\Z$-basis $\{1_0, 1_1, v_1\}$.
There are two projective $\nh(1)$-modules $P_0$ and $P_1$.
Let $\cf(1)$ denote the corresponding quotient of $\cf$ as an $\nh(1)$-bimodule.
It induces an endofunctor of $\mf{D}(\nh(1))$.
Then $\cf(1)(P_0) \cong P_1, \cf(1)(P_1) = 0$, and $\cf(1)^2=0$.
The object $\exp(-\cf(1))$ reduces to a complex $(\mb \ra \cf)$ of two terms.
Then
$$\exp(-\cf(1))(P_0) \cong (P_0 \xrightarrow{\cdot v_1} P_1), \qquad \exp(-\cf(1))(P_1) \cong P_1.$$
So $\Hom(\exp(-\cf(1))(P_0), \exp(-\cf(1))(P_1))=0$, while $\Hom(P_0, P_1) \neq 0$.
Hence, the endofunctor $\exp(-\cf(1))$ is not invertible.

\vspace{.2cm}
\n {\bf Lifting $\sum\limits_{n\ge 0}(-1)^nx^n$ and $\sum\limits_{n\ge 0}x^n$.}

If we remove the nilHecke generators $x_{i,n}, \pa_{i,n}$ from $\nh$, we obtain a simpler diagrammatic algebra $R$ generated by vertical strands $1_n$ and short strands $v_{i,n}$ only.
They satisfy the same relations as in Definition \ref{def nh}.
The construction of the induction bimodule $\cf$, the morphisms $v \in \Hom(\mb,\cf), \wtv \in \Hom^1(\cf,\mb)$ still works in the case of $R$.
The difference is that we do not have the idempotents $e_k, \wte_k$, and the corresponding direct summands $\cfk, \wtfk$ of $\cf^k$.
Like $\exp(\cf), \exp(-\cf)$ using $\cfk, \wtfk$, one can define similar objects in the derived category of $R$-bimodules where $\cfk, \wtfk$ are replaced by $\cf^k$.
The two objects lift $\sum\limits_{n\ge 0}(-1)^nx^n$ and $\sum\limits_{n\ge 0}x^n$, respectively, but these liftings are not invertible functors either.

\end{document}